\documentclass[11pt]{amsart}
\usepackage{geometry}                % See geometry.pdf to learn the layout options. There are lots.
\usepackage{graphicx, color, caption,subcaption}
\usepackage{amssymb}
\usepackage{epstopdf}
\usepackage{pstricks}
\usepackage{pspicture}
\usepackage{psfrag}
\usepackage{epsfig}
\usepackage{tikz}
\newcommand*\circled[1]{\tikz[baseline=(char.base)]{
            \node[shape=circle,draw,inner sep=1.5pt] (char) {#1};}}

\DeclareMathOperator{\arcsec}{arcsec}

\newtheorem{theorem}{Theorem}[section]
\newtheorem{lemma}[theorem]{Lemma}

\newtheorem{remark}[theorem]{Remark}

\newcommand{\eps}{\varepsilon}
\newcommand{\partialp}{\partial_+}
\newcommand{\partialm}{\partial_-}

\newcommand{\cp}{(c^2)'}
\newcommand{\Rzero}{{\mathcal{R}_0}}
\newcommand{\Rone}{{\mathcal{R}_1}}
\newcommand{\Rtwo}{{\mathcal{R}_2}}

\begin{document}
\title[Rarefaction wave for the nonlinear wave system]{Two dimensional Riemann problems for the nonlinear wave system: Rarefaction wave interactions}

	\author[Kim]{Eun Heui Kim}\thanks{The work of Kim was supported by the National Science Foundation under the Grants DMS-1109202, 1615266}
	\address[Kim]{Department of Mathematics and Statistics, California State University, Long Beach, CA 90840, USA.
	Email: EunHeui.Kim@csulb.edu
	}
		
	\author[Tsikkou]{Charis Tsikkou}\thanks{The work of Tsikkou was supported by the National Science Foundation under the Grant DMS-1400168}
	\address[Tsikkou]{Department of Mathematics, West Virginia University, Morgantown, WV 26506, USA.
	Email: tsikkou@math.wvu.edu}
	
\date{\today}

\maketitle

\begin{abstract}
We analyze rarefaction wave interactions of self-similar transonic irrotational flow in gas dynamics for the two dimensional Riemann problems. 
We establish the existence result of the supersonic solution to the prototype nonlinear wave system for the sectorial Riemann data, and study the formation of the sonic boundary and the transonic shock.
The transition from the sonic boundary to the shock boundary inherits at least two types of degeneracies (1) the system is sonic, 
and in addition (2) the angular derivative of the solution becomes zero where the sonic and shock boundaries meet.    


\noindent
AMS: {Primary: 76L05, 35L65; Secondary: 65M06, 35M33.}
\end{abstract}

\section{Introduction}\label{intro}
It is well known that the wave patterns created by initial value problems for the multidimensional compressible Euler system are complicated and challenging to study. As such, numerous works are focused on special initial data that allow us to reduce the space dimensions. 
In particular two-dimensional Riemann problems can be reduced to self-similar problems, which are interesting to study on their own. 
For instance, the flow in self-similar coordinates changes its type; it is supersonic (hyperbolic) in the far-field and becomes mixed near the origin.
A seminal work by Zhang and Zheng \cite{ZZ} illustrated various complicated wave patterns on four sectorial Riemann data for two-dimensional self-similar Euler systems.
Their conjectures are later validated numerically, see for example \cite{KT, LaxLiu, SR}. 
In particular Lax and Liu \cite{LaxLiu} attested the complicated wave patterns including Mach reflections, rolling up and instability of slip lines, and much more.  
They also noted that a general theory of Glimm type for one-dimensional problems would not be likely established for the multidimensional problems. 

While establishing a comprehensive analysis to understand the complicated wave patterns in two dimensional flows with Riemann data still has a long way to go,
there is recent progress describing the solution structures of transonic regular shock reflection problems for various systems,
\cite{CKK,CKK:WRR,CKKmixed, CKK:nlwe, Chen, ChenFang, EL, JKC, KimJDE10, KimCPDE, KimLeeNA13, Morawetz,serrefrei,sever, Zheng1}.
This is an incomplete list of related work specifically on transonic shocks by Riemann problems or by a wedge, and interested readers can refer to the references therein. 
On the other hand, relatively little is known about transonic rarefaction waves.
The computational result by Glimm {\it et al} \cite{Glimm} showed the shock formations from the rarefaction wave interactions.  

This paper addresses the understanding of rarefaction wave interactions and their shock formations in a specific self-similar problem, the nonlinear wave system in two space dimensions.
More precisely, the simple wave created by the planar rarefaction wave becomes a different wave with two families of non-trivial characteristics.
Both families of characteristics merge into becoming sonic near the origin and the type of the system then changes to subsonic, that is the characteristics no longer transport the data any further.
We call this wave a  {\em transient wave}.
On the other hand, there exists a family emanating from this transient wave region, becoming compressive and forming a shock downstream.
Thus the type changes and forms a sonic boundary in some part and a transonic shock in the other.
%In this configuration, similar to other rarefaction wave interaction problems, as discussed in \cite{Glimm} and the references therein, there exists a family emanating from the transient wave region, which becomes compressive and forms a shock.

We note that a similar configuration was studied by \cite{SongZheng} for the pressure gradient system.
They \cite{SongZheng, Zhengbook} called the region -- where a family of characteristics starts on sonic curves and ends on transonic shock curves --  
the  ``semi-hyperbolic'' region. 
 \cite{SongZheng} established the existence of a local solution in a given semi-hyperbolic region, provided a smooth convex boundary and small Riemann data, and \cite{WangZheng} established the local regularity result for this solution. Our work is motivated by the work of \cite{Bang, DaiZhang, SongZheng, WangZheng}.

The nonlinear wave system, which can be considered as wave motions of shallow water and multidimensional $p$-systems, is a reduced system of the compressible Euler system for isentropic, irrotational flow in two space dimensions \cite{CKKmixed,CKK:nlwe}.
The nonlinear wave system can be also considered as a part of an operator splitting scheme in numerics, where the compressible Euler system can be split into the nonlinear wave system (the pressure system) and the pressure-less system (the gradient flow). In fact \cite{Zhengbook} noted that the Euler system can be split into the pressure-gradient system and the pressure-less system, see \cite{Song, Zheng1} and the references therein.
The pressure-gradient system is a special case of the nonlinear wave system.
The pressure-less system is well understood by \cite{SZ}.
Hence if one understands the solution structure of the nonlinear wave system, then one can construct the solution of the Euler system successively by using the splitting method.
Furthermore, there are many similarities on the structures of both the nonlinear wave system and the Euler system, see \cite{CKKmixed, serre}. 
As such, it is crucial to understand the nonlinear wave system in order to study the Euler system.

We focus on the wave patterns created by planar rarefaction waves.  
For the configuration, we impose four sectorial Riemann data $U_i$, $i=1,2,3,4$ for each $ith$ quadrant, see the left figure in Figure~\ref{fig_Rdata}.
\begin{figure}[t]
\begin{subfigure}[t]{0.3\textwidth}
\psfrag{1}[][][0.7][0]{$R_{14}$}
\psfrag{2}[][][0.7][0]{$R_{34}$}
\psfrag{3}[][][0.7][0]{$J_{12}$}
\psfrag{4}[][][0.7][0]{$J_{23}$}
\psfrag{5}[][][0.7][0]{$U_1$}
\psfrag{6}[][][0.7][0]{$U_4$}
\psfrag{7}[][][0.7][0]{$U_3$}
\psfrag{8}[][][0.7][0]{$U_2$}
\begin{center}
\includegraphics[height = 2in,width = 2in]{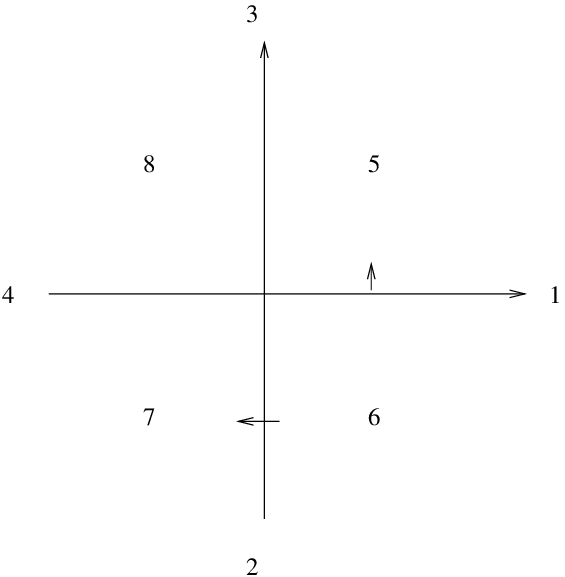}
\end{center}
%\vspace*{-4mm}
\caption{Riemann Data.}
\end{subfigure}
\hspace{1in}
\begin{subfigure}[t]{0.3\textwidth}
\psfrag{C_1}[][][0.7][0]{$C_1$}
\psfrag{C_2}[][][0.7][0]{$C_4$}
\psfrag{12}[][][0.7][0]{$R_{14}$}
\psfrag{13}[][][0.7][0]{$R_{34}$}
\begin{center}
\includegraphics[height = 2in,width = 2in]{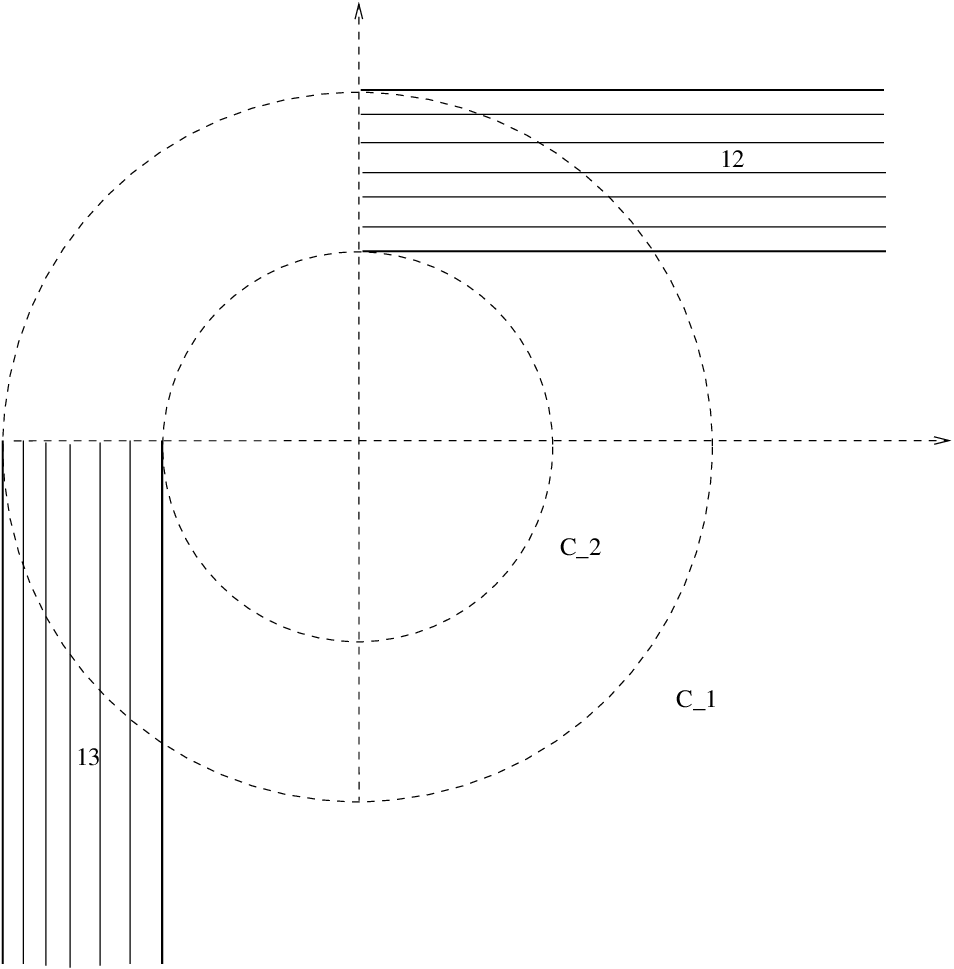}
\end{center}
%\vspace*{-4mm}
\caption{Configuration.}
\label{Configuration_Dec}
\end{subfigure}
\vspace*{-4mm}
\caption{Riemann data and configuration.}
\label{fig_Rdata}
\end{figure}
We consider planar waves with a horizontal rarefaction wave $R_{14}$ created by the constant data $U_1$ and $U_4$, and another vertical rarefaction wave $R_{34}$ created by $U_3$ and $U_4$. The contact discontinuities reside along the positive $y-$axis created by $U_1$ and $U_2$, and along the negative $x-$axis by $U_2$ and $U_3$, which have no effect on the system (the reduction of the system makes the contact discontinuities trivial), see the right figure in Figure~\ref{fig_Rdata}.
The data are symmetric with respect to $x=-y$ and thus it suffices to focus only on $R_{14}$.
Near the locus of the sonic circle (it is the origin for the nonlinear wave system) 
the type of the flow changes and becomes subsonic, creating a transonic shock downstream. 
We formulate the boundary value problem and establish the supersonic solution for the entire hyperbolic region of this configuration. 

Since the change of the type is not known a priori, the problem has two different types of free boundaries: the sonic boundary and the transonic shock boundary.
We show that the sonic boundary in this configuration inherits at least two types of degeneracies.

%%Typically the position of the transonic shock is unknown a priori and hence it gives rise to a free boundary problem.
%%On the other hand the sonic boundary is also unknown as the change of the wave is not known a priori, which makes another type of free boundary. 
%This is a new type of free boundary and little 
%progress has been made to rigorously understand such problems.  
%In this configuration, the shock free boundary created by a family of characteristics 
%has a complicated feature. %Namely the rarefaction wave becomes sonic in part and creates a transonic shock downstream.
%On the other hand the sonic boundary is also unknown as the change of the wave is not known a priori, which makes another type of free boundary. 
%This is a new type of free boundary and little 
%progress has been made to rigorously understand such problems.  
%It is also of interest to understand how the compressive wave is created by the expansion wave. 
%Numerical studies \cite{Tesdall1, Tesdall2} suggest that similar wave patterns are observed for the Mach reflection problem near the Mach stem. 
%We show that the sonic boundary created by the rarefaction wave inherits at least two types of degeneracies.

The first obvious one is that the wave across the sonic boundary becomes degenerate meaning it is neither hyperbolic nor  elliptic. 
The degeneracy of this type is well known as a Tricomi type problem, where the characteristics enter the sonic boundary perpendicularly.  
The Tricomi type degeneracy appears, for example, in the airflow over a wing where the steady subsonic flow creates a supersonic region over the convex surface of the wing creating a shock, see Courant and Friedrichs \cite{CF}.  This is a long standing open problem. 
 Numerical studies \cite{Tesdall1, Tesdall2} suggest that similar wave patterns are observed for the Mach reflection problem near the Mach stem. 
 We establish the existence of the supersonic solution, which becomes sonic, and that the solution and the sonic boundary are $C^1$.

The second one is when a family of characteristics creates a compression wave downstream, the sonic boundary and the transonic shock merge into a point,
% -- by construction we have a smooth solution so that these two transonic boundaries, the sonic and the shock, meet at a point, instead of a cluster of points --
and at that point, the angular derivative of the solution also disappears. To our knowledge, this is a new type of free boundary problem. 
We note that the Tricomi type degeneracy, since the characteristics enter the sonic boundary perpendicularly, implies that the sonic boundary is never the characteristics. Thus we utilize both the directional derivatives along the sonic boundary and the data transported along the characteristics to have the solution and the sonic boundary to be $C^1$. The solution may not be $C^1$ at the point where the sonic and shock boundaries meet, and at that point, the sonic boundary no longer has a Tricomi type degeneracy but the angular derivative of the solution becomes zero. Our results provide an insight to understand how the compressive wave is created by the expansion wave.  

For multidimensional conservation laws, the entropy conditions are insufficient to answer whether our solution is the physically relevant one. However, our solution captures the numerics, see Section~\ref{numerics}. 
This paper provides a framework to establish existence of the supersonic solution suggested by the numerics, and an analysis to understand how the type of the rarefaction wave changes. 
The complete analysis to construct the transonic wave in the entire region including the subsonic region will be discussed in our forthcoming paper.

 The main contributions of this paper are the following. 
 We first formulate the boundary value problem for the self-similar nonlinear wave system.
  We next discuss the wave patterns, monotonicity properties, regularity and existence results.
  The solution will be constructed locally and then assembling the pieces together along the characteristics and the sonic boundary. 
  We show that the sonic boundary created by the transient waves is $C^1$ and is strictly increasing radially.
This sonic boundary is terminated and radially tangential when it merges to the transonic shock downstream. 
Numerical results by using CLAWPACK for certain pressures are presented as well.

 We believe our results will serve as a vehicle for understanding transonic flows in particular the long standing open problem of the flow over the convex wing, and lead to further developments of systematic theories for multi-dimensional conservation laws. Interested readers can refer to the survey paper \cite{SChen10} for the comprehensive references and recent progress in transonic problems.

\section{Description of the problem}
\subsection{Nonlinear wave system: Configuration}
From the compressible Euler system for isentropic flow in two space dimensions,
ignoring the nonlinear velocity terms (assuming low velocities)
and assuming irrotational flow,
we can deduce a simpler system, the nonlinear wave system \cite{CKKmixed};
\begin{equation}\label{nlwesys}
\begin{array}{rcl}
\rho_t + (\rho u)_x + (\rho v)_y &=&0\\
(\rho u)_t+ p_x &=&0\\
(\rho v)_t +p_y &=&0.
\end{array}
\end{equation}
Here
 $\rho(t, x,y)$ is the density, $ u(t,x,y)$ and $v(t,x,y)$ are the $x$ and $y$ components of velocity, respectively, and $p(\rho)$ is the pressure satisfying a polytropic gas law
 \[\frac{dp}{d\rho}  = c^2(\rho) =k \gamma \rho^{\gamma-1}, \]
 with constants $k$ (we let $k=1$ for simplicity), $1<\gamma<\infty$, (typically $1<\gamma<2$, $\gamma=5/3$ is air),
 and a local sound speed $c^2(\rho)$.
 This system can be considered as wave motions of shallow water and multidimensional $p$-systems.
 
We let the momentum $(\rho u, \rho v) =(m,n)$ and use $U$ to denote  $(\rho, m,n)$: 
\[
U=(\rho, m,n)=(\rho, \rho u, \rho v). 
\]

Specifically we have the following Riemann data $U_i=(\rho_i, m_i, n_i)$ where $i=1,2,3,4$ for each quadrant satisfying;
\begin{eqnarray*}
R_{14} :& \rho_1> \rho_4, \ m_1=m_4, \ n_1- n_4= \Phi_{14}\\
J_{12} : &\rho_1=\rho_2, \ m_1=m_2, \ n_1 > n_2\\
J_{23}:  &\rho_2=\rho_3, \ m_2 > m_3, \ n_2=n_3\\
R_{34}: &\rho_3 > \rho_4, \ m_3 -m_4 = -\Phi_{34}, \  n_3=n_4,
\end{eqnarray*}
where
\begin{eqnarray*}
\Phi_{ij} &=& \int^{\rho_i}_{\rho_j} c(s) ds = \frac{2}{\gamma+1} \left( \rho_i^{\frac{\gamma+1}{2}} -   \rho_j^{\frac{\gamma+1}{2}} \right).
\end{eqnarray*} 
That is, 
\begin{eqnarray*}
\rho_1=\rho_2=\rho_3>\rho_4, \ m_1=m_2=m_4 > m_3, \ n_2=n_3=n_4 < n_1.
\end{eqnarray*}
Hence we impose four sectorial data;
\begin{eqnarray*}
U_1&=& (\rho_1, 0, \Phi_{14})^t\\
U_2 &=& (\rho_1, 0,0)^t\\
U_3&=& (\rho_1, -\Phi_{14}, 0)^t\\
U_4&=& (\rho_4,0,0)^t.
\end{eqnarray*}
We let $c_1=c(\rho_1)$ and $c_4=c(\rho_4)$ and denote two sonic circles $C_1=\{(c_1,\theta), 0\le \theta\le 2\pi\}$ and $C_4=\{ (c_4,\theta),0\le \theta\le 2\pi\}$ corresponding to the Riemann data $\rho_1>\rho_4$.

We replace $c^2(\rho)= \gamma p^\kappa$, $\kappa = (\gamma -1)/\gamma$, and write the system \eqref{nlwesys} in self-similar coordinates $\xi=x/t, \eta=y/t$;
 \begin{equation}\label{selfsys}
\begin{array}{rcl}
-\frac{1}{\gamma} p^{-\kappa} (\xi p_\xi  + \eta p _\eta) + m_\xi + n_\eta &=& 0,\\
-\xi m_\xi -\eta m_\eta+  p_\xi  &=& 0,\\
-\xi n_\xi -\eta n_\eta  + p_\eta &=& 0.
\end{array}
\end{equation}
The system can be written in a second order equation 
(by applying $\partial_\xi$ to the second equation and $\partial_\eta$ to the third equation in \eqref{selfsys} and replacing the momentum terms with their derivatives from the first equation in \eqref{selfsys})
\begin{equation}\label{Pself}
p_{\xi\xi} + p_{\eta\eta}=\frac{1}{\gamma}p^{-\kappa}(\xi p_{\xi}+\eta p_{\eta})+\xi \bigg(\frac{1}{\gamma} p^{-\kappa} (\xi p_\xi  + \eta p _\eta)\bigg)_\xi + \eta \bigg(\frac{1}{\gamma} p^{-\kappa} (\xi p_\xi  + \eta p _\eta)\bigg)_\eta,
\end{equation}
which can be written in polar coordinates $r^2=\xi^2 +\eta^2,\theta = \tan \eta/\xi$:
\begin{eqnarray}\label{Pnd}
r^2 \bigg(1 -\frac{r^2}{c^2} \bigg)  p_{rr} +  p_{\theta\theta} +  r(1-2\frac{r^2}{c^2}) p_r + \kappa \frac{r^2}{c^2} \frac{r^2}{p} p^2_{r} =0.
\end{eqnarray}
The system is hyperbolic (supersonic) when $c^2 <r^2$, sonic when $c^2 =r^2$, and elliptic (subsonic) when $c^2 >r^2$.

In the following section we discuss the characteristic equations for the system in the supersonic region.
\subsection{Characteristic equations in the supersonic region}
When the state is hyperbolic, that is $c^2 <r^2$, we let 
\begin{eqnarray*}
\lambda_\pm &=& \pm r\sqrt{\frac{r^2 -c^2}{c^2}}, \quad {\rm or \ simply} \quad \lambda = r\sqrt{\frac{r^2 -c^2}{c^2}},
\end{eqnarray*}
and the positive and negative characteristics derived by integrating $\frac{dr}{d\theta} =\lambda_{\pm}=\pm\lambda,$ respectively, 
so that \eqref{Pnd} reads  
\begin{eqnarray}\label{Pchar}
p_{\theta\theta} - \lambda^{2} p_{rr} &=& -\frac{r}{c^2}(c^2 -2r^2)p_r -\frac{\kappa r^4}{c^2 p} p_r^2.
\end{eqnarray}
In addition, by letting $\partial_{\pm} =\partial_\theta \pm \lambda \partial_r$, we have, as in \cite{SongZheng},
\begin{eqnarray*}
\partial_-\partial_+ p &=& \frak{h}(\partial_- p -\partial_+ p)\partial_+ p\\
\partial_+\partial_- p &=& \frak{h}(\partial_+ p -\partial_- p)\partial_- p,
\end{eqnarray*}
where
\begin{eqnarray*}
\frak{h}&=&\frac{r^2 \cp}{4c^2(r^2 -c^2)} =\frac{r^4\cp}{4c^4\lambda^2},  \quad \cp=\frac{dc^2}{dp} =\gamma  \kappa p^{\kappa -1} = (\gamma-1)p^{\kappa -1}.
\end{eqnarray*}
We denote 
\begin{eqnarray*}
R=\partial_+ p, && S=\partial_-p
\end{eqnarray*}
so that
\begin{eqnarray*}
p_\theta = \frac{1}{2}(R+S), && p_r =\frac{1}{2\lambda } (R-S),
\end{eqnarray*}
and
\begin{eqnarray*}
\partialm R &=& \frak{h}(S-R)R\\
\partialp S &=& \frak{h}(R-S)S,
\end{eqnarray*}
which can be written as
\begin{eqnarray}\label{pRS}
\left( \begin{array}{c}
p \\ R \\ S
\end{array}
\right)_\theta 
+ \left( \begin{array}{ccc}
0 & 0& 0\\ 0 & -\lambda &0 \\ 0&0&\lambda
\end{array}
\right) \left( \begin{array}{c}
p \\ R \\ S
\end{array}
\right)_r &=& 
\left( \begin{array}{c}
\frac{1}{2}(R+S) \\ \frak{h}(S-R)R \\ \frak{h}(R-S)S
\end{array}
\right).
\end{eqnarray}
In the following section, we discuss and formulate the boundary value problem for the system.

\section{Derivations of the boundary value problem and the main result}

\begin{figure}[ht]
\psfrag{A}[][][0.7][0]{$\Xi_1$}
\psfrag{B}[][][0.7][0]{$\Xi_2$}
\psfrag{C}[][][0.7][0]{$\Xi_3$}
\psfrag{C_1}[][][0.7][0]{$C_1$}
\psfrag{C_2}[][][0.7][0]{$C_4$}
\psfrag{2}[][][0.7][0]{$\xi$}
\psfrag{1}[][][0.7][0]{$\eta$}
\psfrag{5}[][][0.7][0]{$\Sigma$}
\psfrag{6}[][][0.7][0]{$\Gamma$-}
\psfrag{7}[l][][0.7][0]{$\Gamma_{12}$}
\psfrag{3}[][][0.7][0]{$\sigma$}
\psfrag{11}[][][0.7][0]{$\Sigma_{34}$}
\psfrag{12}[][][0.7][0]{$R_{14}$}
\psfrag{13}[][][0.7][0]{$U_1$}
\psfrag{14}[][][0.7][0]{$U_4$}
\psfrag{4}[][][0.7][0]{$\sigma_1$}
\psfrag{8}[][][0.7][0]{$\Gamma_{23}$}
\psfrag{9}[][][0.7][0]{$\Gamma_{24}$}
\psfrag{10}[][][0.7][0]{$\Xi_4$}
\psfrag{15}[][][0.7][0]{$\eta=-\xi$}
\psfrag{16}[][][0.7][0]{$\Sigma_0$}
\psfrag{17}[][][0.7][0]{$\mathcal{R}_0$}
\psfrag{18}[][][0.7][0]{$\mathcal{R}_1$}
\psfrag{19}[][][0.7][0]{$\mathcal{R}_2$}
\begin{center}
\includegraphics[height = 3.3in,width = 4in]{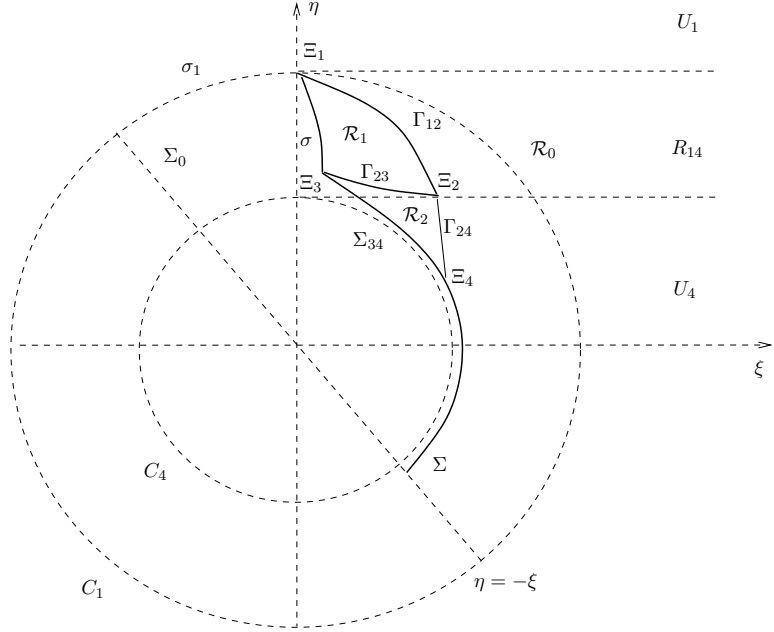}
\end{center}
\vspace*{-4mm}
\caption{Configuration with details.}
\label{fig_config}
\end{figure}

When the rarefaction wave $R_{14}: c^2(p)=\eta^2$, where $p_4\le \eta \le p_1$, enters the sonic circle $C_1$, a new wave pattern is created. 
Specifically, the wave pattern changes from the simple wave to two wave families of non-constant solutions, and these two families of characteristics merge and become sonic near the positive $\eta$-axis. 
We denote the sonic boundary created by these two families of characteristics by $\sigma.$
%We call the new wave pattern, with two wave families of non-constant solutions, the {\em transient double characteristic waves}\footnote{Zheng {\em et al} \cite{SongZheng} call this region semi-hyperbolic.} or {\em transient waves} in short. 
%We note that this wave pattern is not just "double waves", and in addition there 
In addition, there exists a wave family connecting the sonic boundary and the transonic shock boundary.
The transient wave region (consists of two wave families of non-constant solutions) is enclosed by two limiting characteristics and the sonic boundary. 
More precisely,
let $\Gamma_{12}$ be the positive characteristic emanating from $\Xi_1=(r_1,\pi/2)$ to $\Xi_2=(r_2, \theta_2)$, which separates the simple wave $\mathcal{R}_0=\{c^2(p)=\eta^2\}$ and the transient wave.  $\Gamma_{12}$ is completely determined by the simple wave $\mathcal{R}_0$.
$\Xi_2=(r_2, \theta_2)$ is the point at which this characteristic crosses $\eta=c_4$, 
where $r_2=c_4/\sin\theta_2$ and $\theta_2= \sin^{-1}\sqrt{c_4/c_1}$. 
Let $\Gamma_{23}$ be the negative characteristic emanating from $\Xi_2$ to $\Xi_3$ where $\Xi_3=(r_3,\theta_3)$ is 
the point at which the characteristic speed becomes zero, that is, a sonic point.
The transient wave region is the hyperbolic region enclosed by characteristics $\Gamma_{12}$ and $\Gamma_{23}$, and the sonic boundary $\sigma$.

On the other hand, there exists yet another new simple wave created downstream, which is adjacent to the constant state $p=p_4$.
More precisely, when the governing hyperbolic system is reducible to a first order homogeneous system, the new state adjacent to the constant state forms a simple wave, see Courant and Friedrichs \cite{CF}.  
We refer to Dafermos \cite{Dafermos} for details of the general framework of the simple wave.
This simple wave becomes compressive and creates a transonic shock downstream.

In summary, the transonic boundary is created by the transient wave (the sonic boundary) in part, and the simple wave (the shock boundary) in the other.
%two distinct wave patterns are created, separated by specific characteristics which then become subsonic near the origin. 
Hence it is crucial to identify the corresponding characteristics first to separate these different wave regions, in order to formulate the correct boundary problems.

Below we list the notation used for the characteristics and the corner points that we just discussed.
Let
\begin{eqnarray*}
\Gamma_{12} &=& \{ (r,\theta): r=c_1\sin\theta, \ \theta_2\le \theta \le \pi/2\}\\
\Gamma_{23} &=& \{(r,\theta): r=r_{23}(\theta), \ \theta_2\le \theta \le \theta_3\},
\end{eqnarray*}
and 
\begin{eqnarray*}
\Xi_1&=&\left(c_1, \frac{\pi}{2} \right)\\
\Xi_2 &=& \left(\frac{c_1}{\sin\theta_2}, \sin^{-1} \sqrt{\frac{c_4}{c_1}} \right) = \left(\sqrt{c_1c_4},  \sin^{-1} \sqrt{\frac{c_4}{c_1}} \right) \\
\Xi_3 &=& (r_3,\theta_3).
\end{eqnarray*}

The supersonic region  is divided into three regions based on the characteristics where $R=\partial_+p$ and $S=\partial_-p$ are either strictly positive or zero, and denoted by;
\begin{itemize}
\item Rarefaction wave region $\Rzero$: $R>0$ and $S=0$, see Section~\ref{sectionR_0}. 
\item Transient wave region $\mathcal{R}_1$: Both families $R$ and $S$ are non-trivial, enclosed by the positive characteristic $\Gamma_{12}$ emanating from $\Xi_1$ to $\Xi_2$; the negative characteristic $\Gamma_{23}$ from $\Xi_2$ to $\Xi_3$; and the sonic boundary $\sigma$ on which $R=S$. See Sections ~\ref{sectionR_1} and \ref{sectionR_1existence}. 
\item Simple wave region $\mathcal{R}_2$:   $R=0$ and $S>0,$ see Section~\ref{sectionR_2}. 
The boundary of this region consists of  the negative characteristic  $\Gamma_{23}$ emanating from $\Xi_2$ to $\Xi_3$; 
the positive characteristic, denoted by $\Gamma_{24}=\{ (r_{24}(\theta),\theta)\}$, emanating from $\Xi_2$ to the shock $\Sigma$ at the point $\Xi_4$, 
where $r_{24}(\theta)=c(p_4)\sec\bigg(\theta+\arcsec\sqrt{\frac{c(p_1)}{c(p_4)}}-\arcsin\sqrt{\frac{c(p_4)}{c(p_1)}}\bigg)$; and a part of the shock $\Sigma'\subset \Sigma$ from $\Xi_3$ to $\Xi_4$. 
\end{itemize}

In region $\mathcal{R}_1$, we have the following Goursat boundary value problems: 
\begin{eqnarray}\label{eqp}
p_\theta &=& \frac{1}{2}(R+S), \quad p \mid_{\Gamma_{12}}= \gamma^{-1/\kappa} \eta^{2/\kappa}, \quad p\mid_{\Gamma_{23}}= z, \\
\label{eqR} \partialm R &=& \frak{h}(S-R)R, \quad R \mid_{\Gamma_{12}}=  \gamma^{-1/\kappa} \partialp \eta^{2/\kappa}, \quad R\mid_{\Gamma_{23}}= \partialp z,\\
\label{eqS} \partialp S &=& \frak{h}(R-S)S, \quad S \mid_{\Gamma_{12}}= \gamma^{-1/\kappa} \partialm \eta^{2/\kappa}, \quad S\mid_{\Gamma_{23}}= \partialm z,
\end{eqnarray}
where
\begin{eqnarray*}
R \mid_{\Gamma_{12}}=  \gamma^{-1/\kappa} \partialp \eta^{2/\kappa} = \frac{4{\lambda_+}p}{\kappa r}, \quad S \mid_{\Gamma_{12}}= \gamma^{-1/\kappa} \partialm \eta^{2/\kappa} = 0\\
S \mid_{\Gamma_{23}}= \partialm z = z_\theta - z_r r\sqrt{\frac{r^2- c^2}{c^2}} \ge 0, \quad R \mid_{\Gamma_{23}}=\partialp z =  z_\theta+  z_r r\sqrt{\frac{r^2- c^2}{c^2}}.
\end{eqnarray*}
We discuss the compatibility conditions in Section~\ref{sectionR_1}.

In $\mathcal{R}_2$, the positive characteristics, emanating from $\Gamma_{23}$, form an envelope, 
where the positive characteristics satisfy
\begin{eqnarray}\label{pchar}
 \frac{dr}{d\theta} =  \lambda = r\sqrt{\frac{r^2 -c^2(p)}{c^2(p)}} ,\quad r(\theta_i)=r_{23}(\theta_i), 
 \end{eqnarray}
 for each $( r_{23}(\theta_i),\theta_i)\in \Gamma_{23}, \ \theta_2\le \theta_i \le\theta_3,  $ 
and $r=r_{23}(\theta)$ satisfies 
\begin{equation}\label{nch}
\frac{dr}{d\theta} = - \lambda = -r\sqrt{\frac{r^2 -c^2(p)}{c^2(p)}} , \quad r(\theta_2)=c_4/\sin\theta_2, \quad \theta_2<\theta<\theta_3.    
\end{equation}

\begin{remark}
We note that the simple wave does not create a sonic curve. More precisely if it is sonic somewhere, that is $\lambda=0$, in the simple wave region. Then  $S=\partial_-p = p_\theta -\lambda p_r$ and $R=\partial_+p =p_\theta +\lambda p_r$ imply $p_\theta =S=R$, which yields a contradiction: simple wave.
\end{remark}

The family $S>0$ remains to be non-trivial and this is the one that carries the data from $\sigma$ to $\Gamma_{23}$ between the two regions $\mathcal{R}_1$ and $\mathcal{R}_2$.
The other family $R$ becomes zero, and therefore carries no information in this simple wave region $\mathcal{R}_2$. Thus we have
\begin{equation}\label{Simple}
\partial_+ S = -\frak{h} S^2, \quad S\mid_{\Gamma_{23}} = \partial_- z.
\end{equation}
We can write the solution to this Riccati type equation in the following form: Let $(\hat\theta,\hat r)$ a point on the negative characteristic $\Gamma_{23}$ and integrate \eqref{Simple} along the positive characteristic lines in region $\mathcal{R}_2$, then 
\begin{equation}\label{Ssimple}
S=\partial_{-}p(\theta,r)=\frac{\partial_{-}p(\hat \theta,\hat r)}{\partial_{-}p(\hat \theta,\hat r)\displaystyle\int_{\hat \theta}^\theta \frak{h}\ d_{+}\theta+1}.
\end{equation}

We finally state the main theorem. 
\begin{theorem}\label{mainth}
Let $2c(p_4)>c(p_1)$. For a convex $\Gamma_{23}\in C^2$ and the data $(p, R, S) \in C^2(\Gamma_{23})$  satisfying the compatibility conditions at $\Xi_2$ and $\Xi_3$ and on $\Gamma_{23}$, 
there exists a supersonic solution $(p,R, S) \in C^2(\Rone\setminus\Gamma_{23}\cup\sigma) \cap C^1(\Rone\cup\sigma) \cap C^{0,1}(\overline \Rone\setminus\{\Xi_3\})$ satisfying the Goursat boundary problems \eqref{eqp}-\eqref{eqS}.
The solution $(p,R,S)$ creates the sonic boundary $\sigma=\{(\tau(\theta),\theta), \theta_3 <\theta<\pi/2\} \in C^1$
such that $R=S=p_\theta >0$ on $\sigma$ and $\tau'(\theta)>0$ on $\sigma$.
$\tau'(\theta_3)=0$ and $\tau'(\pi/2)=0$.

The sonic boundary $\sigma$ and the transonic shock $\Sigma'$ merge into a point $\Xi_3$ at which the solution $(p,R,S)$ holds 
\begin{equation}
 R(\Xi)=S(\Xi) \rightarrow 0, \quad as \quad \Xi \in \sigma \rightarrow \Xi_3. 
\end{equation}

Furthermore there exists a simple wave creating a transonic shock $\Sigma'$ in region $\Rtwo$. 
\end{theorem}
Our existence result of the supersonic flow is established with the given convex negative characteristics $\Gamma_{23}$ and the data on $\Gamma_{23}$ holding the compatibility conditions. 
In our forthcoming paper, we establish the global transonic solution and provide the scheme to select the correct data and $\Gamma_{23}$. 
 
 In what follows, we use the condition $2c(p_4)>c(p_1)$ and discuss the existence results for each region.
%%%%%%%%%%%%%%%%%%%%%%%%%%%%%%%%%%%%%%%%%%%%%%%%%%%%%%%
\section{Rarefaction wave Region $\mathcal{R}_0$}
\label{sectionR_0}
With $c^2(p)=\gamma p^\kappa = \eta^2$, where $p_4\le p\le p_1$, it is easy to see that
$$p=\frac{1}{\gamma^{1/\kappa}}r^{2/\kappa}(\sin \theta)^{2/\kappa}$$ and 
\begin{eqnarray*}
\lambda = r\sqrt{\frac{r^2-c^2}{c^2}}=r\sqrt{\frac{\xi^2+\eta^2-c^2}{c^2}}=\frac{r\xi}{\eta}=\frac{r\cos \theta}{\sin \theta}.
\end{eqnarray*}
By integrating along the positive characteristic $\Gamma_{12}$ emanating from $\Xi_1,$ $(r_{1},\theta_{1})=(c(p_1),\pi/2)$,
\begin{align*}
\frac{dr}{d\theta}=\lambda =\frac{r\cos \theta}{\sin \theta},
\end{align*}
we find that
\begin{equation}\label{G12}
\Gamma_{12}: \ \ r=c(p_1)\sin \theta, \quad \frac{dr}{d\theta} = c(p_1) \cos\theta =c(p_1)\sqrt{1-\frac{r^2}{c^2(p_1)}}.
\end{equation}
Note that $\Gamma_{12}$ terminates at $\eta=r\sin\theta=c(p_4)$, and thus
$$(r_{2},\theta_{2})= \bigg(\sqrt{c(p_1)c(p_4)}, \arcsin\sqrt{\frac{c(p_4)}{c(p_1)}}\bigg).$$
Hence $\Gamma_{12}$ is completely determined by the rarefaction wave $R_{14}: c^2(p)=\eta^2$, 
where $p_4\le p\le p_1$.
In addition, in region $\mathcal{R}_0$ we have
\begin{align}
R&= \partial_+p= \partial_\theta p +\lambda\partial_r p=\frac{4}{\kappa\gamma^{1/\kappa}}\cos\theta(\sin \theta)^{2/\kappa-1}r^{2/\kappa}>0, \quad \theta\in(\theta_{2}, \pi/2)\label{Rrare}\\
S&=\partial_-p=\partial_\theta p -\lambda\partial_r p =\frac{2}{\kappa}\frac{r^{2/\kappa}}{\gamma^{1/\kappa}}(\sin \theta)^{2/\kappa}\left[\frac{\cos \theta}{\sin \theta}-\sqrt{\frac{r^2\cos^2 \theta}{r^2 \sin^2 \theta}}\right]=0, \quad \theta\in [\theta_{2}, \pi/2].
\end{align}
Note that $R=0$ when $\theta=\pi/2$.

%%%%%%%%%%%%%%%%%%%%%%%%%%%%%%%%%%%%%%%%%%%%%%%%%%%%%%
\section{Transient wave Region $\mathcal{R}_1$}
\label{sectionR_1}
We first discuss many useful properties of the characteristics in the transient wave region $\Rone$, in the same spirit as in \cite{Bang}.
More precisely we discuss the monotonicity and convexity properties of the characteristics, and the monotonicity of $p$ along the characteristics in polar coordinates and cartesian coordinates (different coordinates provide different aspects of the characteristics) in Lemmas~\ref{Lemma1_increasing} --\ref{intersectionpoint}.
We also state a priori bounds at the end of the section, in Lemmas~\ref{apriori}, \ref{apest}. 
\begin{lemma}
\label{Lemma1_increasing}
The hyperbolic solution $p \in C^1$ to the Goursat problem satisfies  $$S=\partial_{-}p>0, \ \ R=\partial_{+}p>0 \ \ \text{in the interior of region} \ \ \mathcal{R}_1.$$
Hence $p_\theta = (R+S)/2>0$ in $\mathcal{R}_1$.
\end{lemma}
\begin{proof}
Let $(\tilde\theta,\tilde r)$ be a point on the positive characteristic $\Gamma_{12}$, and $(\hat\theta,\hat r)$ a point on the negative characteristic $\Gamma_{23}.$ 
Integrate  \eqref{eqR} and \eqref{eqS} along the negative and positive characteristic lines, respectively, to obtain
\begin{align}
R&=\partial_{+}p(\theta,r)=\partial_{+}p(\tilde \theta,\tilde r)\exp\bigg({\int_{\tilde \theta}^\theta \frak{h}(S-R) \ d_{-}\theta}\bigg)>0,\label{R_positive}\\
S&=\partial_{-}p(\theta,r)=\partial_{-}p(\hat \theta,\hat r)\exp\bigg({\int_{\hat \theta}^\theta \frak{h}(R-S) \ d_{+}\theta}\bigg). \label{S_positive}
\end{align}
The data \eqref{Rrare} ensures $R>0$ only.
Thus we check the positiveness of $S$ to complete the proof. 
From \eqref{S_positive}, if $S=0$ at $(\hat \theta,\hat r)$ then $S=0$ along the positive characteristic passing through $(\hat \theta,\hat r)$ in region $\mathcal{R}_1$. Thus $R=S=0$ at a point different from $\Xi_1$ on $\sigma$ which is a contradiction. 
Therefore $S\neq 0$ along $\Gamma_{23},$ possibly excluding the endpoints.
If $S<0$ at $(\hat \theta,\hat r)$  then $R=S<0$ somewhere on the sonic curve which is again a contradiction to $R>0$. 
So we conclude that $S>0$ along $\Gamma_{23}$ for $\theta\neq \theta_{2}, \theta_{3}.$ 
Thus $R, S>0$ and consequently $p_{\theta}>0$ in the interior of region $\mathcal{R}_1.$
\end{proof}
By Lemma~\ref{Lemma1_increasing}, we deduce $R=0$ in the simple wave region $\mathcal{R}_2$.

We next discuss monotonicity properties of the characteristics.
To ease the analysis, we write the characteristics in self-similar coordinates. From $\dfrac{dr}{d\theta}=\pm\lambda$, the characteristics $\eta=\eta(\xi)$ in the $(\xi,\eta)$-plane read
\begin{align}
\label{cartesiancharacteristics}
\frac{d\eta}{d\xi}=\Lambda_{\pm}
=\frac{\xi \eta\pm \sqrt{c^2(\xi^2+\eta^2-c^2)}}{\xi^2-c^2}.
\end{align}
Let the corresponding directional derivatives in the self-similar coordinates be;
$$\frac{dp}{d_{\pm}\xi}=\frac{\partial p}{\partial \xi}+\Lambda_{\pm}\frac{\partial p}{\partial\eta}, \ \ \frac{dp}{d_{\pm}\eta}=\frac{\partial p}{\partial \eta}+\Lambda_{\pm}^{-1}\frac{\partial p}{\partial\xi}.$$
We observe that in region $\mathcal{R}_1\cap \{(\xi,\eta), \xi>0, \eta>0\},$ 
\begin{align}
\xi \Lambda_{-}-\eta&<0,\ \ \ \eta\Lambda_{-}+\xi>0,\label{fornegative}\\
\eta\Lambda_{+}^{-1}-\xi&<0, \ \ \ \xi\Lambda_{+}^{-1}+\eta>0, \label{forpositive}
\end{align}
and $$(c^2-\xi^2)\Lambda_{\pm}^2+2\xi\eta \Lambda_{\pm}+c^2-\eta^2=0,$$ 
which can be solved for $c^2$ to get
\begin{align}
c^2=\frac{(\xi \Lambda_{-}-\eta)^2}{\Lambda_{-}^2+1}, \label{cfornegative} 
\end{align}
or 
\begin{align}
c^2=\frac{(\xi-\eta\Lambda_{+}^{-1})^2}{\Lambda_{+}^{-2}+1} \label{cforpositive}.
\end{align}

We next discuss properties of the characteristics.
\begin{lemma}
\label{Lemma_decreasing/increasing}
The hyperbolic solution $p\in C^2(\Rone)$  to the Goursat  problem  
has the following properties:
\begin{enumerate}
\item Along the negative characteristics $\dfrac{d\eta}{d\xi} = \Lambda_-$ starting from any point on $\Gamma_{12}\setminus \Xi_1$:

(i)  $\displaystyle{\frac{dp}{d_{-}\xi}<0}$,
(ii) $\dfrac{d\Lambda_{-}}{d_{-}\xi}>0$,
(iii) $\Lambda_-<0$, 
and
(iv) ${\dfrac{dp}{d_{-}\eta}>0}$;

\item   Along the positive characteristics $\dfrac{d\eta}{d\xi} =\Lambda_+$ starting from any point on $\Gamma_{23}$:

(v) $\dfrac{dp}{d_{+}\xi}<0$, 
(vi) $\dfrac{d\Lambda^{-1}_{+}}{d_{+}\eta}<0$, 
(vii) $\Lambda_+<0,$ and
(viii) $\dfrac{dp}{d_{+}\eta}>0,$
\end{enumerate}
\end{lemma}
\begin{proof}
(i) From
\begin{equation}
\begin{aligned}\label{S}
0< S=\partial_{-}p&=\bigg(r\cos\theta-r\sin\theta\sqrt{\frac{r^2-c^2}{c^2}}\bigg)\frac{dp}{d_{-}\eta}\\
&=-\bigg(r\sin\theta+r\cos\theta \sqrt{\frac{r^2-c^2}{c^2}}\bigg)\frac{dp}{d_{-}\xi},
\end{aligned}
\end{equation}
and by Lemma~\ref{Lemma1_increasing}, we obtain the strict inequality $$\frac{dp}{d_{-}\xi}<0,$$ everywhere in the region $\mathcal{R}_1\cap \{(\xi,\eta), \xi\geqslant0, \eta>0\}.$

(ii) Differentiating (\ref{cfornegative}) along $\dfrac{d}{d_{-}\xi}$ gives
\begin{align}\label{convexity}
\gamma\kappa p^{\kappa-1}\frac{dp}{d_{-}\xi}=\frac{2(\xi\Lambda_{-}-\eta)(\xi+\eta \Lambda_{-})}{(\Lambda_{-}^2+1)^2}\frac{d\Lambda_{-}}{d_{-}\xi},
\end{align}
thus using (\ref{fornegative}), we conclude that $$\frac{d^2\eta}{d\xi^2}=\frac{d\Lambda_{-}}{d_{-}\xi}>0$$ in $\mathcal{R}_1\cap\{(\xi, \eta), \xi>0, \eta>0\}$, which means that the negative characteristics are convex. 

(iii)-(iv) Evaluate the first equation of (\ref{S}) on $\sigma'=\sigma\setminus\{\Xi_1,\Xi_3\}$ where $c^2=r^2$ to get
\begin{eqnarray}\label{Ssigma}
0< S=r\cos\theta \frac{dp}{d_{-}\eta}.
\end{eqnarray}
This immediately implies $\dfrac{dp}{d_{-}\eta} >0$ on $\sigma'\cap \{(\xi,\eta), \xi> 0, \eta>0\}$. If $\mathcal{R}_1\cap \{(0,\eta), \eta<c(p_1)\}\neq\emptyset$ or $\sigma\cap\{(0,\eta), \eta<c(p_1)\}\neq\emptyset$ then \begin{align}
\label{negativealongsigma2}
\frac{dp}{d_{-}\eta}<0
\end{align} and $\Lambda_{-}=-\dfrac{\xi}{\eta}=0,$ respectively, at the points on the $\eta$-axis. On the other hand, 
\begin{eqnarray*}
0> \dfrac{dp}{d_-\xi} = p_\xi + \Lambda_- p_\eta = \Lambda_-( p_\eta + \Lambda_-^{-1} p_\xi) = \Lambda_- \dfrac{dp}{d_-\eta}
\end{eqnarray*}
in region $\mathcal{R}_1\cap \{(\xi,\eta), \xi\geqslant0, \eta>0\}.$
We show that  $\Lambda_-\neq 0$. 
If not then we have unbounded  $\dfrac{dp}{d_-\eta}$ when $\Lambda_- = 0$ at some point $(r^*,\theta^*)$ on the negative characteristics. However  $c^2=(\eta^*)^2$ when $\Lambda_-=0$, and at the same time $\dfrac{dp}{d_-\xi} = p_\xi<0$. 
 At the points on $\Gamma_{12}$ we know that $c^2=\eta^2$ so by (\ref{cartesiancharacteristics}) the negative characteristics satisfy $\Lambda_{-}=0.$ Hence a convex characteristic has at least two points with $\Lambda_{-}=0$ which leads to a contradiction. Thus, if $\mathcal{R}_1\cap \{(0,\eta), \eta<c(p_1)\}\neq\emptyset$ then $\Lambda_{-}>0,$ $\dfrac{dp}{d_{-}\eta} <0$ and $\dfrac{dp}{d_{-}\xi}<0$ in the region $\mathcal{R}_1\cap \{(\xi,\eta), \xi> 0, \eta>0\}$ which is again a contradiction to the convexity and the behavior of the negative characteristics along $\Gamma_{12}.$  We therefore conclude that the change of type occurs in the first quadrant, $\mathcal{R}_1$ is located in the first quadrant and $\Lambda_{-}< 0$, $\dfrac{dp}{d_{-}\eta} >0$ in the interior of the entire region $\mathcal{R}_1$. In addition, note that $\sigma\setminus\Xi_1\subseteq \{(\xi,\eta), \xi> 0, \eta>0\}.$

(v) By Lemma \ref{Lemma1_increasing}, and 
\begin{align*}
0< R=\partial_{+}p&=\bigg(-r\sin\theta+r\cos\theta \sqrt{\frac{r^2-c^2}{c^2}}\bigg)\frac{dp}{d_{+}\xi}\\
&=\bigg(r\cos\theta+r\sin \theta\sqrt{\frac{r^2-c^2}{c^2}}\bigg)\frac{dp}{d_{+}\eta},
\end{align*}
we conclude that $$\frac{dp}{d_{+}\eta}>0$$ everywhere in the interior of $\mathcal{R}_1$.

(vi) In addition differentiating (\ref{cforpositive}) gives 
\begin{align*}
\gamma \kappa p^{\kappa-1}\frac{dp}{d_{+}\eta}=\frac{2(\xi\Lambda_{+}^{-1}+\eta)(\eta\Lambda_{+}^{-1}-\xi)}{(\Lambda_{+}^{-2}+1)^2}\frac{d\Lambda_{+}^{-1}}{d_{+}\eta}.
\end{align*}
By (\ref{forpositive}), we conclude that 
$$\frac{d^2\xi}{d\eta^2}=\frac{d\Lambda_{+}^{-1}}{d_{+}\eta}<0$$ 
which means that the positive characteristics are concave.

(vii)-(viii) On $\sigma$ (including $\Xi_1$) we have  
\begin{align}
\label{positivealongsigma}
\frac{dp}{d_{+}\xi}<0 \ \, \ \ \text{and} \ \  \Lambda_{+}=-\frac{\xi}{\eta}\le0.
\end{align} 
Thus by a similar argument as before we conclude the claim.
More precisely, from 
\begin{eqnarray*}
0<\dfrac{dp}{d_{+}\eta} = p_\eta + \Lambda^{-1}_+ p_\xi = \Lambda^{-1}_+ (\Lambda_+ p_\eta  + p_\xi) = \Lambda^{-1}_+ \dfrac{dp}{d_+\xi},
\end{eqnarray*}
we show that $\Lambda^{-1}_+\neq 0$ to obtain $\Lambda^{-1}_+<0$ and $\dfrac{dp}{d_+\xi}<0$ in the interior of $\Rone$ by a contradiction argument.
Suppose not. Then $c^2=\xi^2$ and $\dfrac{dp}{d_{+}\eta} = p_\eta>0$ at the contradiction point $(r^*, \theta^*)$ on the positive characteristics. On the other hand for the negative characteristics passing from $(r^*, \theta^*)$ the following hold:
\begin{eqnarray*}
\Lambda_- =-\frac{1}{\tan 2\theta^*} , \quad 0< \partial_-p = \frac{r^*\cos2\theta^*}{\cos\theta^*} \frac{dp}{d_-\eta}.
\end{eqnarray*}
However we have shown that
\[ \dfrac{dp}{d_-\eta} >0.\] 
The contradiction is immediate when $\pi/4 \le  \theta^* \le \pi/2$.
\end{proof}

In the following two lemmas, we discuss the properties of the sonic boundary $\sigma$; in particular the monotonicity and the corner point $\Xi_3$ with the level curve where $\{p=p(\Xi_3)\}$.
\begin{lemma}
\label{pointxi3}
The level curve $\{p=p(\Xi_3)\}$ of the solution $p\in C^1(\Rone)$ 
and the sonic boundary $\sigma$ meet tangentially at $\Xi_3$ at which $$\dfrac{d\eta}{d\xi}=-\dfrac{\xi}{{\eta}}.$$ 
 In addition $R=S$ on the sonic boundary. In particular $R=S>0$ on $\sigma\setminus\{\Xi_1,\Xi_3\}$.
\end{lemma}
\begin{proof}
In region $\mathcal{R}_1\setminus\{\Xi_1,\Xi_3\},$ by Lemma~\ref{Lemma1_increasing},
\begin{align*}
R&=\partial_{+}p=p_{\xi}\bigg(-\eta +\xi\sqrt{\frac{\eta^2+\xi^2-c^2}{c^2}}\bigg)+p_{\eta}\bigg(\xi+\eta\sqrt{\frac{\eta^2+\xi^2-c^2}{c^2}}\bigg)> 0,\\
S&=\partial_{-}p=p_{\xi}\bigg(-\eta -\xi\sqrt{\frac{\eta^2+\xi^2-c^2}{c^2}}\bigg)+p_{\eta}\bigg(\xi-\eta\sqrt{\frac{\eta^2+\xi^2-c^2}{c^2}}\bigg)> 0,
\end{align*}
thus 
\begin{align}
\label{conditiononsonic}
R+S=2(\xi p_{\eta}-\eta p_{\xi})> 0.
\end{align}
Since on the sonic boundary $c^2=r^2$ it is immediate to have $R=S=p_{\theta}=-\eta p_{\xi}+\xi p_{\eta}>0$ on $\sigma\setminus\{\Xi_1,\Xi_3\}$. 
On the boundary $\Gamma_{23}$ (see Section \ref{sectionR_2} for details) we have 
\begin{align}
S=\partial_{-}p&=\frac{2r^3_-(-r'_-)p^{1/\gamma}}{\gamma \kappa}\left[\frac{r_-r''_--r^2_--2r'^2_-}{(r'^2_-+r^2_-)^2}\right]. \label{S1}
\end{align}
If $r_-''$ is bounded (otherwise $S$ might maintain a positive lower bound in the neighborhood of $\Xi_3$ and by (\ref{Ssimple}) that would mean that there is a singularity at $\Xi_3$) then $S\rightarrow 0$ as $\theta\rightarrow \theta_3$ on $\Gamma_{23}$
(note that $\Gamma_{23}$ is prescribed to be convex with the smooth data). 
Thus $\Gamma_{23}$ and the level curve $\{p=p(\Xi_3)\}$ meet tangentially at $\Xi_3$ and $\xi p_{\eta}=\eta p_{\xi}.$

The tangential derivative of $c^2=\eta^2+\xi^2$ along $\sigma =\{(\xi,\eta_\sigma(\xi) )\}$ reads
\begin{align}
\label{sonicboundary}
c_p^2 \ p_{\xi}+c_p^2 \ p_{\eta} \frac{d\eta_\sigma}{d\xi}=2\eta \frac{d\eta_\sigma}{d\xi}+2\xi.
\end{align}
On a level curve $L=\{X=(\xi,\eta_L(\xi))\}$, we have 
\begin{align}
\label{levelcurves}
\frac{dp}{d\xi} (\xi,\eta_L(\xi)) =p_{\xi}+\frac{d\eta_L}{d\xi}p_{\eta}=0.
\end{align}
If $p_{\eta}=0$ at $\Xi_3$ then $p_{\xi}=0$ and thus by (\ref{cartesiancharacteristics}) we have $$\Lambda_{-}=\frac{d\eta_{\sigma}}{d\xi}=-\frac{\xi}{\eta}.$$
If $p_{\eta}\neq 0$ at $\Xi_3$ then $$\Lambda_{-}=\frac{d\eta_{L}}{d\xi}=\frac{d\eta_{\sigma}}{d\xi}=-\frac{\xi}{\eta}.$$ We thus conclude that a level curve meets tangentially the sonic boundary at $\Xi_3.$\end{proof}

\begin{lemma}
\label{intersectionpoint}
The sonic boundary $\sigma=\{(\xi,\eta(\xi))\}$ of the solution $p\in C^2(\Rone)$ is strictly decreasing in the $\xi$ direction. That is $\dfrac{d\eta}{d\xi}<0$ 
everywhere along $\sigma,$ except at $\Xi_1.$
\end{lemma}
\begin{proof}
By Lemma \ref{Lemma1_increasing} we first note that in region $\mathcal{R}_2$ we have $R=0$ and $S>0$ along the positive characteristics and therefore the sonic boundary along which $R=S$ cannot extend below $\Xi_3.$
 
Let $(\xi^{\ast},\eta^{\ast})$ be a point on  $\sigma =\{ (\xi(\eta),\eta)\} $ such that $c^2=(c^{\ast})^2$ and $\dfrac{d\xi(\eta^\ast)}{d\eta} =0$.  
Then, in the neighborhood of this point, $\dfrac{d\eta}{d\xi}$ is unbounded. 
Specifically, $\dfrac{d^2\eta}{d\xi^2}>0$ when $\eta<\eta{\ast}$; 
and $\dfrac{d^2\eta}{d\xi^2}<0$ when $\eta>\eta^{\ast}.$ 
The tangential derivative to $c^2=\xi^2+\eta^2$ along $\sigma=\{ (\xi, \eta(\xi) )\}$ now reads
$\dfrac{d(c^2)}{d\xi}=2\xi+2\eta \dfrac{d\eta}{d\xi}.$ 
We deduce that $\dfrac{d(c^2)}{d\xi}$ is unbounded when $c^2=(c^{\ast})^2,$ $\dfrac{d^2(c^2)}{d\xi^2}>0$ when $c^2<(c^{\ast})^2$ and $$\frac{d^2(c^2)}{d\xi^2}=2+2\bigg(\frac{d\eta}{d\xi}\bigg)^2+2\eta\frac{d^2\eta}{d\xi^2}<0$$ when $c^2>(c^{\ast})^2.$ In the latter case, let $\dfrac{d\eta}{d\xi}=g(\eta)$ then $$ \frac{d^2(c^2)}{d\xi^2}=2+2(g(\eta))^2+2\eta \frac{dg(\eta)}{d\eta}g(\eta)<0.$$ A routine integration yields $(g(\eta))^2<\frac{\text{constant}}{\eta^2}-1,$ which leads to a contradiction because close to $\eta=\eta^{\ast}$ the slope is unbounded. Since the sonic boundary lies in the region $\{(\xi,\eta), \xi\geqslant 0,\  \eta>0\}$ then $\dfrac{d\eta}{d\xi}\leqslant 0.$ 
Note that it would be possible to have a point on $\sigma$ such that $\dfrac{d\xi}{d\eta}=0$ if $\eta^{\ast}=0.$ 
However $\sigma$ should be located only where $0< c(p_4) \le  \eta \le c(p_1)$.

Assume now that there is a point $(\xi_{\ast},\eta_{\ast})$ on $\sigma$ such that $c^2=(c_{\ast})^2$ and $\dfrac{d\eta}{d \xi}=0$ at this contradiction point on the sonic boundary. 
Then from the unboundedness of $\dfrac{d\xi}{d\eta}$, 
we have $\dfrac{d^2\xi}{d\eta^2}<0$ when $\xi>\xi_{\ast}$ and $\dfrac{d^2\xi}{d\eta^2}>0$ when $\xi<\xi_{\ast}.$ 
Again the tangential derivative of $c^2=\xi^2+\eta^2$ on $\sigma=\{ (\xi(\eta),\eta)\}$ 
becomes $\dfrac{d(c^2)}{d\eta}=2\eta+2\xi \dfrac{d\xi}{d\eta}.$
We deduce that $\dfrac{d(c^2)}{d\eta}$ is unbounded when $c^2=(c_{\ast})^2,$ $\dfrac{d^2(c^2)}{d\eta^2}>0$ when $c^2<(c_{\ast})^2$ and $$\frac{d^2(c^2)}{d\eta^2}=2+2\bigg(\frac{d\xi}{d\eta}\bigg)^2+2\xi\frac{d^2\xi}{d\eta^2}<0$$ when $c^2>(c_{\ast})^2.$ In the latter case, let $\dfrac{d\xi}{d\eta}=f(\xi)$ then $$ \frac{d^2(c^2)}{d\eta^2}=2+2(f(\xi))^2+2\xi \frac{df(\xi)}{d\xi}f(\xi)<0.$$ Closely related to the above, in spirit as well as in technique, we end up with a contradiction. Therefore $\dfrac{d\eta}{d\xi}< 0$ along $\sigma,$ except at $\Xi_1.$  

In (\ref{sonicboundary}), notice that if $(c^2)' \ p_{\eta}=2\eta$ then $(c^2)' \ p_{\xi}=2\xi$ and thus $R+S=0$ which does not hold for points along the sonic boundary different from $\Xi_1$ and $\Xi_3.$ We therefore conclude that $(c^2)' \ p_\eta\neq 2\eta$ and does not change sign along $\sigma$ excluding the endpoints, thus 
\begin{align}
\label{slopeofsonic}
\frac{d\eta}{d\xi}=\frac{2\xi-(c^2)' \ p_{\xi}}{(c^2)' \ p_{\eta}-2\eta}.
\end{align}

Let us assume that $(c^2)' \ p_{\eta}>2 \eta$ everywhere along $\sigma,$ then $p_{\eta}$ is positive, and we also know from Lemma \ref{pointxi3} that the sonic boundary satisfies $\dfrac{d\xi}{d\eta}<0,$ in the neighborhood of $\Xi_3.$ On the other hand, in the neighborhood of $\Xi_1,$ since $c^2=\eta^2$ in region $\mathcal{R}_0$ and $c^2=r^2$ on $\sigma,$ we expect that the level curves have positive slopes and thus $p_{\xi}<0$ by (\ref{levelcurves}). This combined with (\ref{slopeofsonic}) gives $\frac{d\xi}{d\eta}>0$ for the sonic boundary, in the neighborhood of $\Xi_1$. This implies that there must be a point along $\sigma,$ different from the endpoints, such that $\frac{d\xi}{d\eta}=0.$ which is a contradiction. We therefore conclude by (\ref{conditiononsonic}) that $(c^2)' \ p_{\eta}<2\eta$ and $(c^2)' \ p_{\xi}<2\xi$ along $\sigma.$   
\end{proof}
Figure \ref{characteristics} depicts the configuration of the characteristic curves in the supersonic region. 
\begin{figure}[ht]
\psfrag{2}[][][0.7][0]{$\Gamma_{12}$}
\psfrag{1}[][][0.7][0]{$\Xi_1$}
\psfrag{15}[][][0.7][0]{$\eta=c(p_1)$}
\psfrag{16}[][][0.7][0]{$\eta=c(p_4)$}
\psfrag{3}[][][0.7][0]{$\Xi_2$}
\psfrag{5}[][][0.7][0]{$\Xi_3$}
\psfrag{6}[][][0.7][0]{$\sigma$}
\psfrag{14}[][][0.7][0]{$\Gamma_{+}$}
\psfrag{13}[][][0.7][0]{$\Gamma_{-}$}
\psfrag{7}[][][0.7][0]{$\mathcal{R}_0$}
\psfrag{11}[][][0.7][0]{$\mathcal{R}_2$}
\psfrag{12}[][][0.7][0]{$\mathcal{R}_1$}
\psfrag{8}[][][0.7][0]{$\Xi_4$}
\psfrag{10}[][][0.7][0]{$\Gamma_{24}$}
\psfrag{9}[][][0.7][0]{$\Sigma'$}
\psfrag{4}[][][0.7][0]{$\Gamma_{23}$}
\begin{center}
\includegraphics[height = 2.9in,width = 2in]{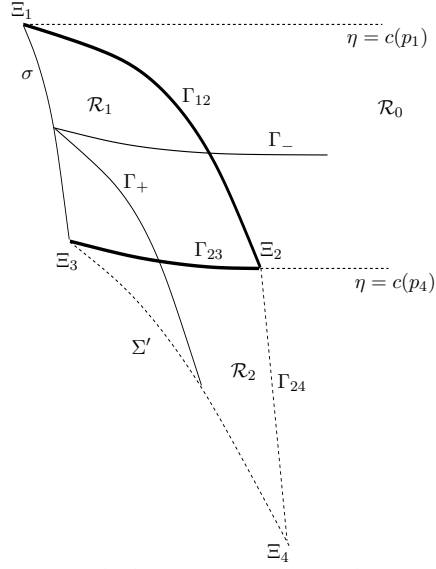}
\end{center}
\vspace*{-4mm}
\caption{Regions and characteristics in the supersonic region.}
\label{characteristics}
\end{figure}
We next state a priori bounds.
\begin{lemma}\label{apriori}
The solution $p\in C^1(\overline {\mathcal{R}_1})$ satisfies
\begin{equation}\label{pbd}
p_4= p(\Xi_2)\le p\le p_1= p(\Xi_1), \quad  in \quad \overline{\mathcal{R}_1}.
\end{equation}  
\end{lemma}
\begin{proof}
By Lemma~\ref{Lemma1_increasing}, we have
\[ p_\theta = (R+S)/2>0.\]
 Thus $p$ is strictly increasing in the $\theta$ direction, and
  $p> p(\Xi_2)=p_4$ in $\mathcal{R}_1$.
 
 Note that the change of type occurs in the first quadrant as the characteristics enter the sonic circle, where $c^2(p)=r^2$, in the first quadrant, 
 and becomes subsonic holding $c^2(p) >r^2$.
 Thus when $\theta\ge \pi/2$ we are now in the subsonic region.
 Therefore $p< p(\Xi_1)=p_1$ in $\mathcal{R}_1$.
 \end{proof}
Similarly, these bounds hold in the simple wave region $\mathcal{R}_2$.
While it is not straightforward to see whether these bounds remain valid in the subsonic region, 
we note that this lemma holds in the entire domain and refer to the forthcoming paper on the transonic problem. 

We next cite estimates from \cite{SongZheng}.
\begin{lemma}\label{apest}
The maximum values of $\partial_{\pm} p$ of the solution $p\in C^2$ are attained on the characteristic boundaries
$\Gamma_{12}\cup\Gamma_{23}.$

Furthermore the solutions $(p, R, S)\in C^1$ satisfy
\[    |t^2 \partial_\pm R|, |t^2 \partial_\pm S| \le C,\]
where $t= \sqrt{ r^2 -c^2(p)}$.
\end{lemma}

\section{The existence result in the transient wave region $\Rone$}\label{sectionR_1existence}
We formulate the Goursat boundary problem to construct the transient wave in $\mathcal{R}_1$.
We first discuss the boundary data on the negative characteristic $\Gamma_{23}$.
Let $f\in C^2(\theta_2,\pi/2)$ be convex, and $g\in C^2(\Gamma_f)$ where $\Gamma_f =\{(f(\theta),\theta): \theta_2\le \theta \le \pi/2 \}$ satisfying
\begin{eqnarray}\label{G23}
\frac{d f}{d\theta} &=& -\lambda(f,g) = - f \sqrt{\frac{f^2- c^2(g(f(\theta),\theta))}{c^2(g(f(\theta),\theta))}} \quad {\rm on }\  \Gamma_f,\\
f^2(\theta) &\ge& c^2(g(f(\theta),\theta)) , \label{fg}\\
f(\theta_2)& =& r_2 = \sqrt{c_1c_4}, \label{r2} 
\end{eqnarray}
and 
\begin{eqnarray}\label{pXi2}
g(\Xi_2)&=&p(\Xi_2),\\
\partialm g(f(\theta),\theta)  &=& g_\theta(f(\theta),\theta) - g_r(f(\theta),\theta) \lambda(f(\theta), g(f(\theta),\theta)) > 0, \quad \theta\in (\theta_2,\pi/2),\label{z23}\\
\partialm g(\Xi_2)&=& 0. \label{SXi2}
\end{eqnarray}
The data $R$ on $\Gamma_f$, denoted by $R[f,g]$, is evaluated from 
$\partial_- R = \frak{h}(f,g) (S-R)R$ along $\Gamma_{f}$ with the initial value $R(\Xi_2)$ stated below, and $\partialm g \mid_{\Gamma_{f}}$ from \eqref{z23}. 
That is
\begin{eqnarray}\label{R23}
 R[f,g](\theta) &=& R(\Xi_2) \exp\int^\theta_{\theta_2} \frak{h}(f(z),g(f(z),z)) (\partialm g(f(z),z) - R(f(z),z)) d z,\\
 R(\Xi_2)&=&\frac{4}{\kappa\gamma^{1/\kappa}}\cos\theta_{2}(\sin \theta_{2})^{2/\kappa-1}r_{2}^{2/\kappa}. \label{RXi2}
 \end{eqnarray}

Additionally, we require that $f$ and $g$ satisfy the following compatibility condition:
 \begin{itemize}
 \item[][G1.]
 There exists $\theta_2 <\theta_3<\pi/2$ such that
 \begin{eqnarray}\label{Xi3}
  \lim _{\theta\rightarrow \theta_3} R[f,g](\theta)  =0,\quad \lim_{\theta\rightarrow \theta_3} \partialm g(f(\theta),\theta) =0.
 \end{eqnarray}
 \end{itemize}
 The compatibility condition [G1.] implies   
 \begin{eqnarray}\label{Xi3p}
 \lim_{\Xi \in \Gamma_f \rightarrow \Xi_3} g_\theta(\Xi) =0,\\
 \lim_{\Xi \in \Gamma_f \rightarrow \Xi_3} (f^2- c^2(g))(\Xi) =0.
 \end{eqnarray}  
% and is written in a weaker form since the actual solution may not be $C^1$ at $\Xi_3$.
We note that due to the conditions that $df/d\theta<0$, $c^2(g) \le f^2$ while $g_\theta>0$ (since $\partial_- g>0$ and $R[f,g]> 0$), we have immediate bounds for $f$ and $g$
\begin{eqnarray}\label{fgbounds}
 c_4^2 < c^2(g) \le f^2 \le c_1 c_4.
 \end{eqnarray}
We construct a solution that creates a sonic boundary where $R=S$ becomes zero as the solution approaches $\Xi_3$ on the sonic boundary.

\subsection{Local existence results}
We first discuss local existence results.
Let $W=(p, R, S)$, and write the system to 
\begin{eqnarray}\label{GoursatBP}
W_\theta + A W_r &=& B,
\end{eqnarray} 
where $A=diag (0, -\lambda, \lambda)$ and $B=(\frac{1}{2}(R+S),  \frak{h} (S-R)R,\frak{h}(R-S)S)$.
The eigenvalues are $\lambda_0=0$, $\lambda_{-} = -\lambda$, and $\lambda_+=\lambda$, 
and the corresponding eigenvectors are $l_0=(1,0,0)$, $l_-=(0,1,0)$ and $l_+=(0,0,1)$. 
Let $f,g$ satisfy \eqref{G23} -- \eqref{RXi2} and [G1.] for some $\theta_3\in (\theta_2, \pi/2)$. 
Thus we have the Goursat boundary value problem \eqref{GoursatBP} with the following boundary data
\begin{eqnarray*}
p^+=p \mid_{\Gamma_{12}}= \gamma^{-1/\kappa} \eta^{2/\kappa}, &\quad & p^- =p\mid_{\Gamma_{23}} =g, \\
R^+ = R \mid_{\Gamma_{12}}=  \gamma^{-1/\kappa} \partialp \eta^{2/\kappa}=\frac{4{\lambda_+}p}{\kappa r}, &\quad &R^-= R\mid_{\Gamma_{23}} = R(\Xi_2) \exp\int^\theta_{\theta_2} \frak{h} (\partialm g -R^-) d_-\theta ,\\
S^+ = S\mid_{\Gamma_{12}} =0, &\quad& S^-=S\mid_{\Gamma_{23}}= g_\theta - g_r \lambda(f, g),
\end{eqnarray*}
which we write  $W^+ =(p^+, R^+, S^+) = W\mid_{\Gamma_{12}}$ and $W^-=(p^-,R^-, S^-)=W\mid_{\Gamma_{23}}$.
Clearly $W^+(\Xi_2)= W^-(\Xi_2)$. 

Furthermore, we have
\begin{eqnarray}\label{sonicdata}
p_\theta&=& R=S, \quad {\rm on } \quad \sigma.
\end{eqnarray}

By checking the compatibility conditions of the Goursat boundary value problems  \cite{Bang, DaiZhang, SongZheng},
we establish the local existence result. 
\begin{theorem}\label{localexistence}
Let the Riemann data satisfy $2c(p_4)> c(p_1)$. 
For a given $\Gamma_{23}\in C^2$ convex, and $W=(p, R, S)\in C^1$ on $\Gamma_{23}$ satisfying \eqref{G23} -- \eqref{RXi2}, and the compatibility condition [G1.], there exists a solution 
$W \in C^1 ({\mathcal{R}_1}(t_0))$ to the system \eqref{GoursatBP}, where 
${\mathcal{R}_1}(t_0) = \{ \sqrt{r^2 -c^2(p)}\ge t_0\}$, for $t_0>0$.
\end{theorem}
\begin{proof}
Let $W^{+}=W\mid_{\Gamma_{12}}$ and $W^-=W\mid_{\Gamma_{23}}$.
Then the compatibility conditions at $\Xi_2$ are the following.  
\begin{eqnarray}\label{comp1}
W^+\mid_{\Xi_2} &=& W^-\mid_{\Xi_2}\\
\frac{1}{\lambda_--\lambda_0} \left( l_0 \frac{dW^- }{d\theta} -\frac{1}{2}(R+S)  \right)\mid_{\Xi_2} &=&\frac{1}{\lambda_+-\lambda_0} \left( l_0 \frac{dW^+ }{d\theta} -\frac{1}{2}(R+S)  \right)\mid_{\Xi_2}.
\end{eqnarray}
The second condition may be rewritten in the form
\[ \frac{dp^+}{d\theta}+\frac{dp^-}{d\theta} = \partial_+ p + \partial _-p.\]
Hence the Goursat problem has a local solution near $\Xi_2$. 

Next we establish a local solution to an initial boundary value problem. Let the initial position be $I=\{(\frak{r}(\theta),\theta)\}$ which is to be determined, and $W_I=W(\frak{r}(\theta),\theta)$. Let $X= I \cap\Gamma_{12}$ and $Y= I \cap\Gamma_{23}$ be the points where the initial and boundary positions meet. Now the compatibility conditions are as follows:
for $i=0,-$, we have 
\begin{eqnarray}
\frac{1}{\lambda_+ -\lambda_i} \left( l_i \frac{dW^+}{d\theta} -  l_i B \right)\mid_X &=&\frac{1}{\frak{r}' -\lambda_i} \left( l_i \frac{dW_I}{d\theta} - l_i B \right) \mid_X;
\end{eqnarray}
and for $i=0,+$, we have
\begin{eqnarray}
\frac{1}{\lambda_- -\lambda_i} \left(l_i \frac{dW^-}{d\theta} - l_i B \right)\mid_Y &=&\frac{1}{\frak{r}' -\lambda_i} \left( l_i \frac{dW_I}{d\theta} - l_i B \right) \mid_Y.
\end{eqnarray}
The initial position $I\in C^1$ is chosen to be the constant level set of $r^2 -c^2(p)$ so that $\frak{r}'\neq \lambda_{\pm}$,  which allows us to match the compatibility conditions. 
Thus we have the local existence result.
\end{proof}

In order to establish the global solution to the entire region $\mathcal{R}_1$, which is  
enclosed by $\Gamma_{12}$, $\Gamma_{23}$ and $\sigma$, we first discuss the regularity results near the sonic boundary due to \cite{WangZheng}. 
We note however that the regularity result is limited to strictly positive $R$ and $S$. 
Hence the estimates depend on $\delta_0$ where $R,S\ge \delta_0>0$.
 Equipped with these regularity results near the sonic boundary, noting that $p_\theta>0$, and the characteristics entering the sonic boundary in the radial direction and never parallel to the tangential direction of the sonic boundary, the sonic boundary $\sigma$ is estimated by the gradient of the pressure which then leads to the existence of a $C^1$ solution and $\sigma \in C^1$.

\subsection{Regularity results}
We discuss the H\"older gradient estimates near the sonic boundary.
The estimates are established by \cite{WangZheng} for the pressure gradient system provided that $\partial_\pm p$ are strictly positive, and under certain smoothness of the solution and the sonic boundary.
While our result relies on \cite{WangZheng}, we provide insights which signify the estimates and their consequences. 

We first change the coordinate system to flatten the sonic boundary as it was done in \cite{WangZheng}.
This new coordinate system brings a couple of technical advantages. The first obvious one is that it simplifies the geometry of the sonic boundary.
The next one is not immediate: the new coordinates enable us to derive the corresponding system that provides estimates on $t^\beta | R_r|, t^\beta |S_r|$, where $t=\sqrt{r^2 -c^2(p)}$, for sufficiently small $0\le t\le t_0$ and uniformly bounded $R, S$ such that $R, S\ge \delta >0$, where $1<\beta=\beta(t_0, \delta)<2$. This is essential to establish the sonic boundary to be in $C^1$.

Since the estimates depend on the strict positiveness of $R, S$, we consider the region $\Rone$ excluding small neighborhoods of $\Xi_1$ and $\Xi_3$, where $\delta$ is the distance from these corner points, see Figure~\ref{corners2}.  
Let $\Rone[\delta]$ be neighborhoods of $\Xi_1$ and $\Xi_3$ with $\delta>0$ distance away from these points 
where $\Gamma_-^\delta$ is the negative characteristic with $\delta =dist (\Xi_1, \Gamma_-^\delta)$, and $\Gamma_+^\delta$ is the positive characteristic with $\delta =dist (\Xi_3, \Gamma_+^\delta)$. 

\begin{figure}[ht]
\psfrag{2}[][][0.7][0]{$\Gamma_{12}$}
\psfrag{1}[][][0.7][0]{$\Xi_1$}
\psfrag{15}[][][0.7][0]{$\eta=c(p_1)$}
\psfrag{16}[][][0.7][0]{$\eta=c(p_4)$}
\psfrag{3}[][][0.7][0]{$\Xi_2$}
\psfrag{5}[][][0.7][0]{$\Xi_3$}
\psfrag{14}[][][0.7][0]{$\Gamma_{+}^{\delta}$}
\psfrag{13}[][][0.7][0]{$\Gamma_{-}^{\delta}$}
\psfrag{4}[][][0.7][0]{$\Gamma_{23}$}
\begin{center}
\includegraphics[height = 1.5in,width = 2.5in]{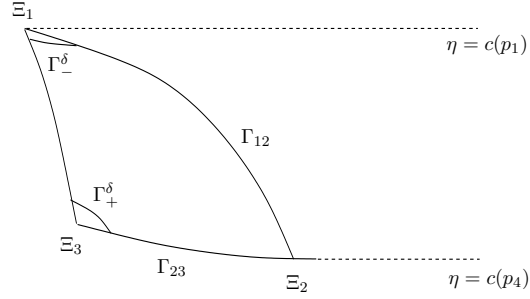}
\end{center}
\vspace*{-4mm}
\caption{$\Rone\setminus \Rone[\delta]$, the region $\Rone$ excluding the small neighborhoods of $\Xi_1$ and $\Xi_3.$ }
\label{corners2}
\end{figure}

The H\"older gradient estimates are established in the region $\Rone\setminus \Rone[\delta]$, in particular near the sonic boundary.
Let $p$ be smooth enough to derive the characteristic equations for the derivatives (for simplicity we may assume $p \in C^3(\mathcal{R}_1)$ for now, which can be weaken).
Let $t=\sqrt{r^2 -c^2}$, and consider the new coordinate system $(r,t)$.
From simple calculations, 
\[\partial_\theta = \frac{-(c^2)' p_\theta}{2t} \partial_t, \ \
\partial_r = \partial_r + \frac{2r -(c^2)' p_r}{2t} \partial_t,\]
 the system in $(r,\theta)$ coordinates becomes
\begin{eqnarray*}
p_t &=& -\frac{2t}{(c^2)'}\\
R_t +\frac{2t \lambda}{(c^2)'S +2r \lambda} R_r &=& \frac{2t}{(c^2)' S + 2r \lambda} \frak{h}(R-S)R,\\
S_t -\frac{2t \lambda }{(c^2)' R - 2r \lambda} S_r &=& \frac{2t}{(c^2)' R -  2r \lambda} \frak{h}(S-R)S,
\end{eqnarray*}
where
\[ \lambda(t,r)=\frac{tr}{\sqrt{r^2 -t^2}}, \quad \frak{h}(t,r)= \frac{r^2 (c^2)'}{4t^2(r^2 -t^2)}.\]
The corresponding characteristics read
\begin{eqnarray*}
\frac{dr_-}{dt} =\frac{2t \lambda}{(c^2)'S + 2r\lambda}, && \frac{dr_+}{dt} =\frac{- 2t \lambda}{(c^2)' R - 2r\lambda}.
\end{eqnarray*}
and the last two equations can be rewritten as
\begin{eqnarray}\label{Rchar}
\frac{dR}{d_{-}t}&=R_t+\dfrac{dr_{-}}{dt}R_r =& \frac{2t}{(c^2)' S + 2r \lambda}\frak{h}(R-S)R, \\
\frac{dS}{d_{+}t}&=S_t+\dfrac{dr_{+}}{dt}S_r = & \frac{2t }{(c^2)' R -  2r \lambda} \frak{h}(S-R)S.\label{Schar}
\end{eqnarray}
As discussed in \cite{WangZheng}, the following equations will be useful to establish the estimates.

Let
\begin{eqnarray}\label{V}
 V=\frac{1}{S}-\frac{1}{R},
 \end{eqnarray}
and obtain
\begin{eqnarray}\label{Vteq}
V_t &=& \mu_0 \frac{V}{t} +\mu_1(r,t) R_r t^2 +\mu_2(r,t) S_r t^2
\end{eqnarray}
where
\begin{eqnarray*}
\mu_0(r,t)&=& \frac{ 4t^2 \frak{h} ((c^2)' + r\lambda V)}{((c^2)' + 2r\lambda S^{-1})((c^2)' - 2r\lambda R^{-1})}\\
\mu_1(r,s)&=& -\frac{2r}{R^2\sqrt{r^2 -t^2}} \frac{1}{(c^2)'S + 2r\lambda} \\
\mu_2(r,t)&=&  -\frac{2r}{S^2\sqrt{r^2 -t^2}} \frac{1}{(c^2)'R - 2r\lambda}.
\end{eqnarray*}
Let
\begin{eqnarray}\label{G}
G&=&\frac{dR}{d_{+}t} - \frac{dR}{d_{-}t} =-2 t^2 e(r,t) R_r  \\
H&=&\frac{dS}{d_{+}t} - \frac{dS}{d_{-}t} =-2t^2 e(r,t) S_r, \label{H}
\end{eqnarray}
and 
\begin{eqnarray*}
e(r,t)&=& \frac{r}{\sqrt{r^2 -t^2}}\left( \frac{1}{(c^2)' R - 2r \lambda} + \frac{1}{(c^2)' S + 2r \lambda}  \right).
\end{eqnarray*}
Evaluate and write
\begin{eqnarray}
\frac{dG}{d_{-}t}
&=& \left(\frac{2}{t} +l(r,t)\right) G + \frac{\tilde f_1(r,t)}{2t} G -\frac{\tilde f_2(r,t)}{2t} H + f_3(r,t) t+ f_4(r,t)t^2,\label{Geq}
\end{eqnarray}
where
\begin{eqnarray*}
l(t) &=& \frac{t}{r^2 -t^2} + \frac{1}{((c^2)' R-2r\lambda)((c^2)' S + 2r\lambda)} \\
&& \times 
\left(  \frac{ (c^2)''}{(c^2)' (R+S)}\frac{2t}{(c^2)'} \left[  \left[ R ((c^2)' S + 2r\lambda)^2 + S((c^2)' R-2r\lambda)^2 \right]
- \frac{(c^2)' r^2(R-S)^2}{\sqrt{r^2 -t^2}}  \right] \right. \\
&& \left. -  \frac{4t^2 \frak{h} r^2 (R-S)}{\sqrt{r^2 -t^2}} 
-\frac{2r^4( (c^2)'(R-S) - 4r\lambda) }{(r^2 -t^2)^{3/2} }
- 4t  \left(\frac{r^2 t^2}{r^2 -t^2}  -\frac{r^2t^4}{(r^2 -t^2)^{2}} \right)
\right),
 \end{eqnarray*}
 and
\begin{eqnarray*}
\tilde f_1(r,t)&=&\frac{4t^2\frak{h}R}{(c^2)' S + 2r \lambda} + \frac{ 4t^2 \frak{h} (R-S)}{ (c^2)' S + 2r \lambda}\\
\tilde f_2(r,t)&=& \frac{4t^2\frak{h}R}{(c^2)' S + 2r \lambda},\\
f_3(r,t) &=&  -e(r,t) \frac{R(R-S)}{(c^2)' S + 2r \lambda} 
\left( \frac{4t^2 \frak{h} (c^2)'' S(R-S) }{2\lambda((c^2)' S + 2r \lambda)} 
+\frac{(c^2)'' r^2 (R-S)}{2\lambda(r^2 -t^2)} + \frac{2r (c^2)'}{ r^2 -t^2} -\frac{2r^3 (c^2)'}{(r^2 -t^2)^2}\right)\\
f_4(r,t)&=& - e(r,t) \frac{4t^2 \frak{h} R(R-S)}{((c^2)' S + 2r\lambda)^2} \left( \frac{2r}{\sqrt{r^2 -t^2}}  - \frac{2rt^2}{(r^2 -t^2)^{3/2}}  \right).
\end{eqnarray*}
Similarly we also have
\begin{eqnarray}
\frac{dH}{d_{+}t} 
&=& \left(\frac{2}{t} +l(r,t)\right) H -\frac{\tilde g_1(r,t)}{2t} G + \frac{\tilde g_2(r,t)}{2t} H + g_3(r,t) t+ g_4(r,t)t^2,\label{Heq}
\end{eqnarray}
where
\begin{eqnarray*}
\tilde g_1(r,t)&=& \frac{4t^2\frak{h}S}{(c^2)' R- 2r \lambda} \\
\tilde g_2(r,t)&=&\frac{4t^2\frak{h}S}{(c^2)' R- 2r \lambda}  +\frac{4t^2\frak{h}(S-R)}{(c^2)' R- 2r \lambda},\\
g_3(r,t) &=&  - e(r,t)  \frac{S(S-R)}{(c^2)' R- 2r \lambda} \left(  \frac{ 4t^2 \frak{h} (c^2)''  R(R-S)}{2\lambda ((c^2)' R -  2r \lambda)} 
+\frac{2r (c^2)'}{r^2-t^2} +  \frac{(c^2)'' r^2 (R-S)}{2\lambda(r^2-t^2)}  - \frac{ 2r^3(c^2)'}{ (r^2 -t^2)^2} \right) \\
g_4(r,t)&=& -e(r,t) \frac{ 4t^2 \frak{h} S(S-R)}{((c^2)' R -  2r\lambda)^2} \left(-\frac{2r}{\sqrt{r^2 -t^2}}  + \frac{2r t^2}{(r^2 -t^2)^{3/2} }  \right).
\end{eqnarray*}
We now establish, in the same spirit as in \cite{WangZheng}, the following regularity result.
\begin{theorem}\label{regularity}
For given Riemann data  $2c(p_4)> c(p_1)$, $t_0>0$ sufficiently small, and $\delta>0$, 
there exist a positive constant $C$ and $\beta\in(1,2)$ depending only on $t_0, \delta$, the Riemann data and $\max_{\Gamma_{12}\cup\Gamma_{23}} (R, S)$, such that 
the solutions $R, S\in C^1(\mathcal{R}_1\setminus\Rone[\delta])$ satisfy 
\begin{eqnarray}\label{Holderestimates}
 |R_t|,  |S_t|,  t^\beta |R_r|,  t^\beta |S_r| \le C \delta^{-1}, \quad \forall t\le t_0.
\end{eqnarray}
Moreover, $R, S$ and $t^{-1}(R-S)$ are uniformly continuous in $\mathcal{R}_1\setminus \Rone[\delta]$, 
that is, $R-S= O(\sqrt{r^2 -c^2(p)})$, and consequently $\sigma \cap \overline {\mathcal{R}_1\setminus \Rone[\delta]} \in C^1$.
\end{theorem}
\begin{proof}
Recall that 
\begin{eqnarray*}
p_\theta = \frac{R+S}{2},&&p_r =\frac{R-S}{2\lambda^{-1}} =\frac{R-S}{t} \frac{\sqrt{r^2 -t^2}}{2r}.
\end{eqnarray*}
Hence in order to have the sonic line to be in $C^1$, we show that $(R-S)/t$ is uniformly bounded, and $R, S$ and $(R-S)/t$ are uniformly continuous in 
$\mathcal{R}_1\setminus \Rone[\delta]$,

For each fixed $t=t_b>0$, we consider the level curve $\{ r^2 -c^2(p(r,\theta)) =t^2_b \}$ where $\theta=\theta(r)$.
We then have 
\[ \theta'(r)=\frac{2r- (c^2)' p_r}{ (c^2)' p_\theta}.\]
Since $R$ and $S$ are positive and bounded in region $\mathcal{R}_1$, we have $p_\theta =(R+S)/2$ positive and bounded.
Thus $\theta'(r)$ is well-defined on each level curve.

Recall that $V=1/S -1/R$ and $V$ satisfies \eqref{Vteq}.
We have $m_i$, $i=0,1$ positive constants such that 
$|(\mu_0-1)/t|\le m_0 \delta^{-1} $, and $|\mu_j|\le m_1 \delta^{-3}$ where $j=1,2$, 
 in region $\mathcal{R}_1\setminus\Rone[\delta]$.  
Hence we can write equation \eqref{Vteq}  in the form
\begin{eqnarray*}
&&\partial_t \left( \frac{V(r,t)}{t} \exp \left(\int^{t_b}_t \frac{\mu_0-1}{\tau} d\tau \right)\right) = \left[ \mu_1(r,t) R_r t +\mu_2(r,t) S_r t \right]\exp \left(\int^{t_b}_t \frac{\mu_0-1}{\tau} d\tau\right).
\end{eqnarray*}
Integrating the last equation from $t$ to $t_b$, we have
\begin{align*}
\frac{V}{t}\mid_{t=t_b} &-  \frac{V(r,t)}{t} \exp \left(\int^{t_b}_t \frac{\mu_0-1}{\tau} d\tau \right) \\
&= \int^{t_b}_t \left[ \mu_1(r,\tau ) R_r \tau +\mu_2(r,\tau) S_r \tau \right]\exp \left(\int^{t_b}_\tau \frac{\mu_0-1}{\sigma} d\sigma\right)d\tau,
\end{align*}
which implies
\begin{align} \nonumber
\exp(-m_0 \delta^{-1} & t_b) \left|\frac{V(r,t)}{t}\right| \\
&< \exp \left(-\int^{t_b}_t \frac{\mu_0-1}{\tau} d\tau \right)  \left|\frac{V(r,t)}{t} \right|\\ \nonumber
&\le \left|\frac{V}{t} \mid_{t=t_b} \right| + \int^{t_b}_t | \mu_1(r,\tau ) R_r \tau  +\mu_2(r,\tau) S_r  \tau|  \exp \left(\int^{t_b}_\tau \left| \frac{\mu_0-1}{\sigma} \right| d\sigma\right)d\tau\\
&\le  \left|\frac{V}{t} \mid_{t=t_b} \right| +  \exp(m_0 \delta^{-1} t_b) m_1 \delta^{-3} \int^{t_b}_t  \tau |R_r|   + \tau | S_r|   d\tau. \label{Vt}
\end{align}
Thus we first establish the estimates on $ tR_r $ and $tS_r $ to obtain the bound of $V/t$.

Recall $G=-2t^2 e(r,t) R_r$ and $H=-2t^2 e(r,t) S_r$, satisfying \eqref{Geq} and \eqref{Heq} respectively.
It is important to observe that the singular terms are with $G$ and $H$ of order $O(t^{-1})$ while the remaining terms are of order $O(t)$.

We first establish the estimates of $G$ and $H$ from \eqref{Geq} and \eqref{Heq}. 
We define 
\[ M=\max\{ t^{-\beta} |G|, t^{-\beta} |H|\},\]
where $\beta>1$ and $M$ are positive constants to be determined.

From \eqref{Geq} we write 
\begin{align*}
\frac{d}{d_-t} \bigg( G(r,t) &\exp\int^t_{t_b} -\left(\frac{2}{\tau} + l(r_-(\tau),\tau)    \right) d\tau \bigg)\\
=& 
\left( \frac{\tilde f_1(r_-(t),t )}{2t } G -\frac{\tilde f_2(r_-(t),t)}{2t} H + f_3(r_-(t),t) t+ f_4(r_-(t),t)t^2\right) \\
& \times\exp\int^t_{t_b} -\left(\frac{2}{\tau} + l(r_-(\tau),\tau)   \right) d\tau.
\end{align*}
Integrate the last equation along the minus characteristic curve passing through $(r_0, \eps_0)$ to get
\begin{eqnarray*}
&&G(r,t) \left(\frac{t_b}{t}\right)^{2}  \exp\int^{t_b}_t  l(r_-(\tau),\tau)  d\tau \\
&=&
G(r, t_b) \\
&&+ \int^{t_b}_t \left(-\frac{\tilde f_1(r_-(\tau),\tau )}{2\tau } G + \frac{\tilde f_2(r_-(t),\tau )}{2\tau}  H - f_3(r_-(t),\tau) \tau - f_4(r_-(t),\tau) \tau^2\right) \\
&& \times \left(\frac{t_b}{\tau}\right)^{2}  \exp\int^{t_b}_\tau  l(r_-(\sigma),\sigma)  d\sigma d\tau.
\end{eqnarray*}
We then deduce
\begin{eqnarray*}
&&G(r,t) \left(\frac{t_b}{t}\right)^{2}  \exp\int^{t_b}_t  l(r_-(\tau),\tau) d\tau\\
&\le &
G(r, t_b) \\
&&+  \int^{t_b}_t \left[ \left(\frac{|\tilde f_1(r_-(\tau),\tau )| }{2}  + \frac{| \tilde f_2(r_-(\tau),\tau )|}{2 } \right) M \tau^{\beta-1} \right.\\
 &&\left. + |f_3(r_-(\tau),\tau)| \tau+  |f_4(r_-(\tau),\tau)|\tau^2 \right]  \left(\frac{t_b}{\tau}\right)^{2}  \exp\int^{t_b}_\tau  l(r_-(\sigma),\sigma)  d\sigma d\tau.
\end{eqnarray*}
Observe that for $t_0>0$, there exist positive constants $L_0$ and $F_0$ such that 
$|l(r,t)|\le L_0 \delta^{-1} $, $\tilde f_1, \tilde f_2\le F_0$ and $ |f_3|, |f_4| \le F_1 \delta^{-1} $ for $t \le t_0$.
Hence for $t\le t_b \le t_0$, we get from the last inequality
\begin{eqnarray*}
&&G(r,t) \left(\frac{t_b}{t}\right)^{2}  \exp (L_0 \delta^{-1} (t- t_b)) \\
&\le &
G(r, t_b) \\
&&+ t_b^2  \int^{t_b}_t \left( F_0 M \tau^{\beta -3}    +  F_1 \delta^{-1} (\tau^{-1} +1) \right) \exp(L_0 \delta^{-1} (t_b-\tau))     d\tau\\
&\le & G(r, t_b) \\
&&+ t_b^{2} \exp(L_0\delta^{-1}  t_b) \frac{1}{2-\beta} F_0M  \left( \frac{1 }{t^{2-\beta}} -\frac{1}{t_b^{2-\beta}} \right)  
+ t_b^{2} \exp(L_0\delta^{-1}  t_b) F_1\delta^{-1}  ( \ln(t_b) -\ln t +t_b-t).
\end{eqnarray*}
We then deduce 
\begin{eqnarray*}
G(r,t) t^{-\beta} &\le & G(r,t_b) t_b^{-2} t^{2-\beta}  \exp (L_0 \delta^{-1}  t_b)
+ \exp(2 L_0 \delta^{-1}  t_b) \frac{1}{2-\beta} F_0 M  \left( 1 -\frac{t^{2-\beta}}{t_b^{2-\beta}} \right)  \\
&&+ \exp( 2 L_0\delta^{-1}  t_b) F_1\delta^{-1}  ( \ln(t_b) -\ln t +t_b-t)t^{2-\beta}.
\end{eqnarray*}
Thus for $1<\beta<2$, by choosing $t_b\le t_0$ sufficiently small if necessary, we then have
\begin{eqnarray*}
\frac{ \exp(2 L_0 \delta^{-1}  t_b)  F_0}{2-\beta}  \left( 1 -\frac{t^{2-\beta}}{t_b^{2-\beta}} \right) &\le & C_0 <1
\end{eqnarray*}
for all $t \le t_b$. Hence we now have established the bound of $M$;
\begin{eqnarray*}
M &\le & \frac{C_1}{1-C_0},
\end{eqnarray*}
where 
\[ C_1 =\max_{\{ t\le t_b\} }[ G(r,t_b) t_b^{-2} + F_1 ( \ln(t_b) -\ln t +t_b-t) \exp (L_0\delta^{-1}  t_b)] \exp (L_0 \delta^{-1}  t_b)  t^{2-\beta}.\]
This uniform bound (similarly to $H$) holds for any $1<\beta<2$, and immediately gives 
\[t^\beta |R_r|, t^\beta|S_r| \le M.\]
Now with this uniform bound, from inequality \eqref{Vt} for $V/t$, noting $|e|\le E_0 \delta^{-1}$, we have 
\begin{eqnarray*}
 \left|\frac{V(r,t)}{t}\right| 
&\le & \exp(m_0  \delta^{-1} t_b) \left|\frac{V}{t} \mid_{t=t_b} \right| +  \exp(2 m_0 t_b) m_1 \delta^{-3}  \int^{t_b}_t  \tau| R_r|   + \tau| S_r|   d\tau\\
&\le & \exp(m_0  \delta^{-1}  t_b) \left|\frac{V}{t} \mid_{t=t_b} \right| +  \exp(2 m_0 t_b) m_1\delta^{-3} 
\int^{t_b}_t  \tau^{\beta -1} M   \frac{1}{e(r,\tau)}  d\tau\\
&\le&  \exp(m_0  \delta^{-1} t_b) \left|\frac{V}{t} \mid_{t=t_b} \right|+ \exp(2 m_0   \delta^{-1}t_b) \frac{m_1 \delta^{-2}  M}{\beta E_0} (t_b^{\beta} -t^{\beta})\\
&\le & M_1 \delta^{-2}.
\end{eqnarray*}
Hence we get
\[ \left|\frac{R-S}{t}\right|  \le M_1 RS \delta^{-2} \le CM_1.\]

We next show that $R$, $S$ and $V/t$ are uniformly continuous.
We now have
\begin{eqnarray*}
\left|\frac{dR}{d_-t}\right| &=& \left| \frac{2t^2 \frak{h} R}{ (c^2)' S + 2r\lambda^{-1}} \frac{R-S}{t} \right| \le CM_1\\
\left|\frac{dS}{d_+t}\right| &=& \left| \frac{2t^2 \frak{h} S}{ (c^2)' R- 2r\lambda^{-1}} \frac{S-R}{t} \right| \le CM_1,
\end{eqnarray*} 
for some constant $C>0$ uniformly in $t$.
Hence integrate the last inequalities along the negative and positive characteristics respectively to obtain
\begin{eqnarray*}
|R(r_1, 0) -R(r_0, t_0)| &\le & CM_1 t_0\\
|S(r_2, 0) -S(r_0, t_0)| &\le & CM_1 t_0,
\end{eqnarray*}
Since $R=S$ along the sonic line and both $R$ and $S$ are continuous inside region $\mathcal{R}_1$, we have
\begin{eqnarray*}
|R(r_1, 0)- R(r_2, 0)|, |S(r_1, 0)- S(r_2, 0)|&\le& 2CM_1 t_0 + |R(r_0, t_0) - S(r_0, t_0)|\rightarrow 0
\end{eqnarray*}
as $|r_1-r_2|\rightarrow 0$. Thus $R$ and $S$ are continuous on the sonic line and uniformly continuous in $\mathcal{R}_1\setminus \Rone[\delta]$.

Next, integrating \eqref{Vteq} from $0$ to $t_b$ where $t_b$ is arbitrary chosen,
we have
\begin{align*}
\frac{V}{t}\mid_{t=t_b} &-  \frac{V}{t}\mid_{t=0} \exp \left(\int^{t_b}_0 \frac{\mu_0-1}{\tau} d\tau \right)\\
 &= \int^{t_b}_0 \left[ \mu_1(r,\tau ) R_r \tau +\mu_2(r,\tau) S_r \tau \right]\exp \left(\int^{t_b}_\tau \frac{\mu_0-1}{\sigma} d\sigma\right)d\tau,
\end{align*}
 which then becomes
 \begin{eqnarray*}
&& \left|  \frac{V}{t}\mid_{t=t_b} -  \frac{V}{t}\mid_{t=0}  \right|\\
&\le & \left| \frac{V}{t}\mid_{t=0}  \right|  \left( \exp  \int^{t_b}_0 \left|\frac{\mu_0  \delta^{-1} -1}{\tau}\right|  d\tau -1 \right)
+ m_1 \delta^{-3}  M  \int^{t_b}_0 \tau^{\beta-1}  \exp \left(\int^{t_b}_\tau \frac{\mu_0  \delta^{-1} -1}{\sigma} d\sigma\right)d\tau\\
&\le & 2M_1 e^{m_0  \delta^{-1} t_b} t_b +  \frac{m_1\delta^{-3} M e^{m_0  \delta^{-1} t_b} t_b^{\beta}}{\beta} \le M_2 t_b.
\end{eqnarray*}
Hence with $M_2 t_b \le \eps/4$ (take $\eta>0$ such that if $|r_1-r_2|\le \eta$ then 
$\left|  \frac{V}{t}(r_1, t_b) -  \frac{V}{t}(r_2,t_b)  \right|\le \eps/4$)
we have
\begin{eqnarray*}
 \left|  \frac{V}{t}(r_1, 0) -  \frac{V}{t}(r_2,0)  \right|
 &\le& \left|  \frac{V}{t}(r_1, 0) -  \frac{V}{t}(r_1, t_b)  \right| 
 + \left|  \frac{V}{t}(r_1, t_b) -  \frac{V}{t}(r_2,t_b)  \right|\\
 && +\left|  \frac{V}{t}(r_2, b) -  \frac{V}{t}(r_2,0)  \right| \le \eps
\end{eqnarray*}
for $|r_1-r_2|\le \eta$.
Thus $V/t$ is also uniformly continuous in $\mathcal{R}_1\setminus \Rone[\delta]$.

Therefore we have established the claim.
\end{proof}
\subsection{The supersonic solution in region $\Rone$}
The existence result in the entire region $\Rone$ is established in two steps.
%First we show Lemma~\ref{exlemma} to establish the existence result in region $\Rone\setminus \Rone[\delta]$, see Figure~\ref{corners2}, away from the points $\Xi_1$ and $\Xi_3$.
%Next, in Lemma~\ref{excorners}, the solution in $\Rone[\delta]$ near $\Xi_1$ and $\Xi_3$ are constructed by adjusting the correct choices of the boundary data $f$ and $g$ on $\Gamma_{23}$ that satisfy the compatibility condition [{\bf G1.}].
We first show Lemma~\ref{exlemma} to establish the sonic boundary $\sigma$ where $R=S$.
We next show that $\overline\sigma\cap \Gamma_{12}=\Xi_1$ and $\overline\sigma\cap \Gamma_{23}=\Xi_3$.

%
%We now discuss the following lemma, in the same spirit as in \cite{Bang, DaiZhang}, which will play a key role in establishing the supersonic solution and a $C^1$ sonic boundary. 
The proof of lemma~\ref{exlemma} is inspired by the work in  \cite{Bang, DaiZhang}.
\cite{DaiZhang} and later \cite{Bang},  established the global existence result for the degenerate hyperbolic system of the pressure gradient equation, where the pressure becomes zero at the origin which makes the system degenerate only at the origin and hyperbolic elsewhere. Furthermore since the pressure gradient system is quasilinear, the type of the system (whether supersonic or subsonic) must be also identified. 
While our system may appear to have  similar technical difficulties as in \cite{Bang, DaiZhang}, we note that it is not straightforward to apply their result to our case. We also note that the sonic boundary will be determined by the choice of the data on $\Gamma_{23}$ and thus the compatibility conditions at $\Xi_1$ and $\Xi_3$ will play crucial roles in selecting the correct data on $\Gamma_{23}$ to find the supersonic solution in the entire region $\Rone$.

Let $u=r^2 -c^2(p)$ so that $u=0$ on the sonic boundary $\sigma$.
The characteristic equations in $u$ become
\begin{eqnarray}\label{up}
\partial_+ u = u_\theta + \lambda_+ u_r &=& 2\frac{r^2}{c(p)} \sqrt{u} -(c^2)' R, \\
\partial_- u = u_\theta +\lambda_- u_r &=& -2\frac{r^2}{c(p)} \sqrt{u} -(c^2)' S. \label{um}
\end{eqnarray}
Thus we have
\begin{eqnarray}\label{utheta}
u_\theta &=& -\frac{(c^2)'}{2} (R+S),\\
u_r &=& 2r -  \frac{(c^2)'c}{r} \left(\frac{R-S}{\sqrt{u}}\right).\label{ur}
\end{eqnarray}

%For a given $\delta>0$ small, we first consider $\Rone\setminus \Rone[\delta]$, see Figure~\ref{corners2}.
\begin{lemma}\label{exlemma}
%Let $\delta>0$ be fixed and consider $\Rone\setminus \Rone[\delta]$. 
For each $d$, where $0<d \le d_m < c^2(p_1)$, there exists a level curve 
$l_\tau=\{(\tau(\theta), \theta): u(\tau(\theta),\theta)=d\} \subset {\mathcal{R}_1}$ and $  \tau=\tau(\theta)\in C^1$ such that 
the following hold:
\begin{enumerate}
\item there exists $U= (u, R, S) \in C^1(D_\tau)$ satisfying the Goursat boundary problem \eqref{eqR},\eqref{eqS},\eqref{utheta},
\item Only the positive family of characteristics intersects with $\Gamma_{23}$. Only the negative family of characteristics intersect with $\Gamma_{12}$,
\end{enumerate}
where %$D_\tau \subset {\mathcal{R}_1}\setminus \Rone[\delta]$ 
$D_\tau \subset {\mathcal{R}_1}$ 
is the closed domain enclosed by $\Gamma_{12}$, $\Gamma_{23}$ and $l_\tau$.
\end{lemma}
\begin{proof}
Let $L$ be the set where $d\in L$ satisfying the assertions (1)-(2) in this lemma.
Since the system for $W=(p, R, S)$ is equivalent to the new system for $U=(u,R, S)$, the local existence result from the system of $W=(p, R,S)$ ensures the existence of $d_0\in L$ (that is $L\neq \emptyset$) and $[d_0, d_m]\in L$ (since $u_\theta<0$).
Hence we only need to show that $\inf L=0$ to establish the claim.
We show the claim by contradiction. Suppose that $\inf L =d^*>0$. 
The proof consists of two main steps.
By extracting a limit to first show that $d^*\in L$, and then show that this $d^*$ violates the infimum assertion.

Since $d^*$ is assumed to be the infimum of $L$, there exists a monotone decreasing sequence $\{d_n\}\subset L$ satisfying $\lim_{n \rightarrow \infty} d_n = d^*$. Then for each $d_n$, there exists $\tau_n=\tau_n(\theta)$ which satisfies the assertions (1)-(2) where $l_n$ is enclosed in the domain bounded by $\Gamma_{12}$ and $\Gamma_{23}$, that is $l_n=\{(\tau_n(\theta), \theta), \theta^n_{23}\le \theta\le \theta^n_{12} \}$ where $l_n\cap \Gamma_{12}=(r^{n}_{12},\theta^n_{12})$ and $l_n\cap \Gamma_{23}=(r^n_{23}, \theta^n_{23})$. We let $D_n$ be the domain enclosed by $l_n$, $\Gamma_{12}$ and $\Gamma_{23}$.
The uniqueness of the local existence results ensure the monotonicity of $D_n$ such that $D_n\subset D_{n+1}$. 
Thus we let $l^*$ be the graph of $\tau^*(\theta)$ where $\lim_{n \rightarrow \infty} \tau_n(\theta)= \tau^*(\theta)$, for all $\theta^*_{23}\le \theta\le \theta^*_{12}$.  Let $D^*$ be the closed region enclosed by $l^*$, $\Gamma_{12}$ and $\Gamma_{23}$.

Hence in $D^{*}\setminus l^*$, there exists the solution $U=(u,R,S)$ satisfying (due to Lemmas 5.1-5.6 and Theorems 6.1, 6.2)
\begin{eqnarray}
d^* \le u\le c^2(p_1)-c^2(p_4),\\
|u_\theta + \lambda_{\pm} u_r| \ge c_0>0,\\
u_\theta \le -c_0\\
\|U\|_{C^{1,\beta}(D^*\setminus l^*)}\le c_1,
\end{eqnarray}
where $c_0, c_1>0$ depends only on $p_1,p_4$ and $d^*$.
Since the tangential derivative of $u$ along $l_\tau$ is
\begin{eqnarray}
u_\theta + \tau' u_r=0.
\end{eqnarray}
we now obtain 
$\|\tau\|_{C^{1,\beta}}\le C.$
By using the standard compactness argument we have $\|\tau\|_{C^{1,\alpha}}\le C_1,$ for any $0<\alpha<\beta$.
Repeating a similar compactness argument we have $U=(u,R,S)\in C^{1,\alpha}$ with the uniform bound established from the regularity results. Thus we extend $U$ in $D^*$ to satisfy the governing system.

Now observe that $l^*$ is the constant level set of $u$, and thus
\begin{equation}\label{lV}
 u_\theta + (\tau^*)'(\theta) u_r=0, \quad {\rm on } \quad l^*.
 \end{equation}
 On the other hand 
 \begin{equation}
 u_\theta  + \lambda_{\pm} u_r \neq 0\quad {\rm on } \quad l^*.
 \end{equation}
 Thus 
 \begin{equation}
  (\tau^*)'(\theta)\neq  \lambda_{\pm},\quad {\rm on } \quad l^*.
 \end{equation}
 Therefore $d^*\in L$.
 
 Next we show that there exists $\eps>0$ small such that $d^*-\eps\in L$ to contradict $\inf L = d^*$.
 Since we have shown that $d^* \in L$, by the local existence result, we extend the solution in $C^1$ by solving the initial boundary value problems where the initial value on $l^*$ satisfies the compatibility conditions on $\Gamma_{23}\cap l^*$ and $\Gamma_{12}\cap l^*$. 
 Furthermore the solution is unique (since the data is prescribed uniquely) and $u_\theta<0$, the corresponding solution determines the level curve where $u=d^*-\eps>0$ for some small $\eps>0$ satisfying the assertions (1)-(2) in this lemma.
 Hence repeating the same argument as to $d^*$ we can show $d^*-\eps\in L$, which is a contradiction.
 
 Therefore $\inf L =0$ and this completes the proof.
\end{proof}
%
%Lemma~\ref{exlemma} shows that there exists the solution $(p, R, S)$ in the region 
%$\Rone\setminus \Rone[\delta]$ satisfying the Goursat boundary value problem with the given $f$ and $g$ satisfying the compatibility conditions on $\Gamma_{23}$.
%We next construct the solution near the corner points $\Xi_1$ and $\Xi_3$. 

Lemma~\ref{exlemma} implies the existence of the sonic boundary $\sigma$ where $R=S$. Due to the monotonicity properties of the characteristics and the solution, the sonic boundary $\sigma$ is connected, bounded and will not form a closed loop.
We now check whether this sonic boundary satisfies the compatibility conditions at $\Xi_1$ and $\Xi_3$.

%
% if $\sigma$ does not intersect with $\Gamma_{12}$ is inside of $\Rone$ that is $\sigma\cap \Gamma_{12}=\emptyset$ and $\sigma\cap\Gamma_{23}=\emptyset$, we then have $R=S>0$ on $\partial\sigma$.
%We then have an open segment of $\sigma$ which we can  consider as the boundary value problem with the characteristic $\Gamma_-^{\delta}$ and a part of $\Gamma_{12}$ near $\Xi_1$ to find the corresponding solution which gives rise the sonic boundary for the open segment. Hence $\sigma$ must be connected with $\Gamma_{12}$.

We first consider $\Xi_3.$
%
%Let $\sigma[\delta]=(\tau(\theta),\theta)$ be the sonic boundary constructed by solving the Goursat boundary problem in the region $\Rone\setminus \Rone[\delta]$.
%By construction we find $\eps_0>0$ such that $\inf \tau' =\eps_0>0$ on $\sigma_\delta $.
%%Since $\Xi_1$ can be treated similarly we only discuss a neighborhood of $\Xi_3$. 
%
%Let  $X_1 \in\sigma[\delta] \cap\partial \Rone\setminus \Rone[\delta]$, and rewrite $\delta =\delta_0$. 
%Let $\Gamma_+^{\delta_1}$ be the positive characteristic connecting $X_1$ and $Y_1 \in \Gamma_{23}$. 
%Then since $R(X_1) = S(X_1) \ge q_0>0$ for some constant and $s' \ge \eps_0>0$ on $\sigma[\delta_0]$ there exist 
If $\sigma$ meets at $X_0\in\Gamma_{23}$ where $X_0\neq \Xi_3$.
This then violates the data at $X_0$ since $S\neq R$ on $\Gamma_{23}$.
So if $\sigma$ meets somewhere on $\Gamma_{23}$ it must be only at $\Xi_3$.
Now if $\sigma$ does not intersect with $\Gamma_{23}$, that is the sonic boundary is terminated before it reaches to $\Gamma_{23}$ and $dist(\overline\sigma,\Gamma_{23})=\delta>0$. 
In that case, since $R$ and $S$ must be positive in the region $\Rone$, we have $R=S>0$ on $\overline\sigma$.
In particular we can find the positive characteristic $\Gamma_+^\delta$ connecting the boundary point, denoted by $X_1$, of $\overline\sigma$, and $\Gamma_{23}$.
Hence we can treat this as a new boundary value problem where the boundaries consist of $\Gamma_+^{\delta}$ and a segment, denoted by $\Gamma_{23}^\delta$, of $\Gamma_{23}$ from $\Xi_3$ to $Y_1=\Gamma_+^\delta \cap \Gamma_{23}$. 
We note that the data on $\Gamma_+^\delta$ is now prescribed with the given data on $\Gamma_{23}\setminus \Gamma_{23}^\delta$.
 
 Note that $R=S>0$ on $\overline\sigma$. Thus we may write $R(X_1) = S(X_1)=q_0>0$ and $\tau'\ge \eps_0>0$ for some constants $q_0, \eps_0$ (however note that these constants $q_0, \eps_0$ depend on the choice of the data $g$ on $\Gamma_{23}$). Hence we find 
$0<\delta_1< \delta_0$ and $\eps_1$ such that 
we can extend $\tau'(\theta) \ge \eps_1>0$, where $\eps_1<\eps_0$,  in a small neighborhood of $X_1$, denote $B_{\delta_1}(X_1)$ to be the circle centered at $X_1$ with radius $\delta_1$.
We now find the positive characteristics $\Gamma_{+}^{\delta_2}$ connecting $Y_2\in\Gamma_{23}$ and 
$\partial B_{\delta_1/2}(X_1)$, see Figure~\ref{construction}. 
If there exists $\delta_1$ such that $B_{\delta_1}(X_1)$ covers the entire neighborhood of $\Xi_3$ then our choice of $g$ on the data to prescribe $\Gamma_{23}$ needs to be adjusted (scale down the strength on $S$).

We next construct the solution by solving the initial and boundary value problems in the regions \circled{1} and \circled{2} enclosed by $\Gamma_+^{\delta_1}$, $\Gamma_{-}^{\delta_1}$, $\Gamma_+^{\delta_2}$, and $\Gamma_{23}$, respectively, so that we find the sonic boundary $\sigma[\delta_1]$, where 
$ R=S\ge q_1>0$ on $\sigma[\delta_1]$ for some constant $q_1\le q_0$. 

Repeat this process by finding $\eps_n\rightarrow 0$ to construct $\sigma[\delta_n]$ for each $\eps_n>0$ and find sequences $\{X_n\}$, $\{Y_n\}$, 
and $q_n$, where $R=S\ge q_n$ on $\sigma[\delta_n]$.
By the construction (these sequences are monotone and bounded) we know that there exist limits for these sequences. Let $X_*, Y_*, q_*$ be the limits of these sequences respectively. 
Note that the tangential derivative on $r^2 -c^2(p)$ along $\sigma$ is
\begin{eqnarray*}
(c^2)'(p_\theta +\tau' p_r) = 2r \tau'
\end{eqnarray*}
which implies $R=S=p_\theta$ on $\sigma$ and 
\begin{eqnarray*}
p_\theta =q_n \rightarrow q_*=0,
\end{eqnarray*}
as $\eps_n \rightarrow 0$.

Hence we are left to check whether $X_*=Y_*$. 
%$X_n \rightarrow \Xi_3$ and $Y_n\rightarrow \Xi_3$.
%If $X_n \rightarrow \Xi_3$ then it is clear by the construction that $Y_n \rightarrow \Xi_3$.

Suppose $ X_*\neq Y_*$. %Then we must have $Y_n \rightarrow Y_*\neq \Xi_3$.
We have two possibilities either (1) $X_*$ surpasses below $\Gamma_{23}$, 
or (2) the sequence $\{X_n\}$ is terminated before it reaches to $\Gamma_{23}$. 
(We may treat this as a shooting method, that is, case (1) is overshoot, and case (2) is undershoot).
We note however by the construction (we have adjusted the data $g$ so that $\delta_i$ can be chosen appropriately), $X_*$ should not be located below $\Gamma_{23}$.
Hence the sonic boundary constructed by this sequence is then terminated before it reaches $\Gamma_{23}$.
In other words we need to readjust the choice of the data $g$. 
Therefore there exists data $g$ such that the limits of $X_n$ and $Y_n$ match, and denote the limit by $\Xi_3$. 

For the case $\Xi_1$.
As before, since $R>S=0$ on $\Gamma_{12}$, if $\sigma$ meets somewhere on $\Gamma_{12}$, it must be at $\Xi_1$.
Now if $\sigma$ does not intersect with $\Gamma_{12}$, as we did for $\Xi_3$, we formulate a boundary value problem where the boundaries are now with the corresponding characteristics $\Gamma_-^\delta$ and $\Gamma_{12}^\delta$. Repeating a similar argument as we did for $\Xi_3$ and noting that the data on $\Gamma_{12}$ is given (independent of the data on $\Gamma_{23}$) while the data on $\Gamma_-^\delta$ depends on the data on $\Gamma_{23}$ we find the correct data so that it matches the compatibility condition at $\Xi_1$.

Note that at $\Xi_3$ both datum on $\Gamma_+^\delta$ and $\Gamma_{23}^\delta$ depend on the choice of $g$, while at $\Xi_1$ it is only on $\Gamma_-^\delta$ that depends on $g$. 

Therefore we establish the following lemma.
\begin{lemma}\label{excorners}
There exist $f\in C^2(\theta_2,\theta_3)$ and $g\in C^2(\Gamma_{23})$ where $\Gamma_{23} =\{(f(\theta),\theta): \theta_2\le \theta \le \theta_3\}$, satisfying
\eqref{G23} -- \eqref{RXi2} and [G1.] with $\theta_2 <\theta_3<\pi/2$, such that
the Goursat boundary value problem has the solution $(p, R, S)\in C^1(\Rone)\cap C^0(\Rone \cup\{\Xi_1, \Xi_3\})$ 
that satisfies the following:
\begin{enumerate}
\item
there exists the sonic boundary $\sigma =\{(\tau(\theta),\theta), \tau'>0 \}$ 
where $R=S>0$ on $\sigma$.
\item
$\sigma$ is terminated at $\Xi_3=(f(\theta_3),\theta_3)$ with $\lim_{\theta \rightarrow \theta_3} \tau'(\theta)=0$. That is $\Xi_3=\overline \sigma\cap \overline \Gamma_{23}$.
\item
$\sigma\in C^1$
\item 
$R(\Xi)=S(\Xi) \rightarrow 0$ as  $\Xi \in \sigma \rightarrow  \Xi_k$, $k=1,3$.
\end{enumerate}
\end{lemma}
\begin{figure}[ht]
\psfrag{1}[][][0.7][0]{$X_1$}
\psfrag{2}[][][0.7][0]{$\Gamma^{\delta_1}_{+}$}
\psfrag{3}[][][0.7][0]{$Y_1$}
\psfrag{4}[][][0.7][0]{$Y_3$}
\psfrag{5}[][][0.7][0]{$\Xi_3$}
\psfrag{6}[][][0.7][0]{$X_3$}
\psfrag{7}[][][0.7][0]{$\circled{2}$}
\psfrag{8}[][][0.7][0]{$\circled{1}$}
\psfrag{9}[][][0.7][0]{$\Gamma^{\delta_3}_{+}$}
\psfrag{10}[][][0.7][0]{$X_2$}
\psfrag{11}[][][0.7][0]{$B_{\delta_1}(X_1)$}
\psfrag{12}[][][0.7][0]{$\delta_1$}
\psfrag{13}[][][0.7][0]{$Y_2$}
\psfrag{14}[][][0.7][0]{$\Gamma_{23}$}
\psfrag{15}[][][0.7][0]{$\delta_1/2$}
\psfrag{16}[][][0.7][0]{$\Gamma^{\delta_2}_{+}$}
\psfrag{17}[][][0.7][0]{$\delta_2$}
\psfrag{18}[][][0.7][0]{$B_{\delta_2}(X_2)$}
\begin{center}
\includegraphics[height = 2.5in,width = 3in]{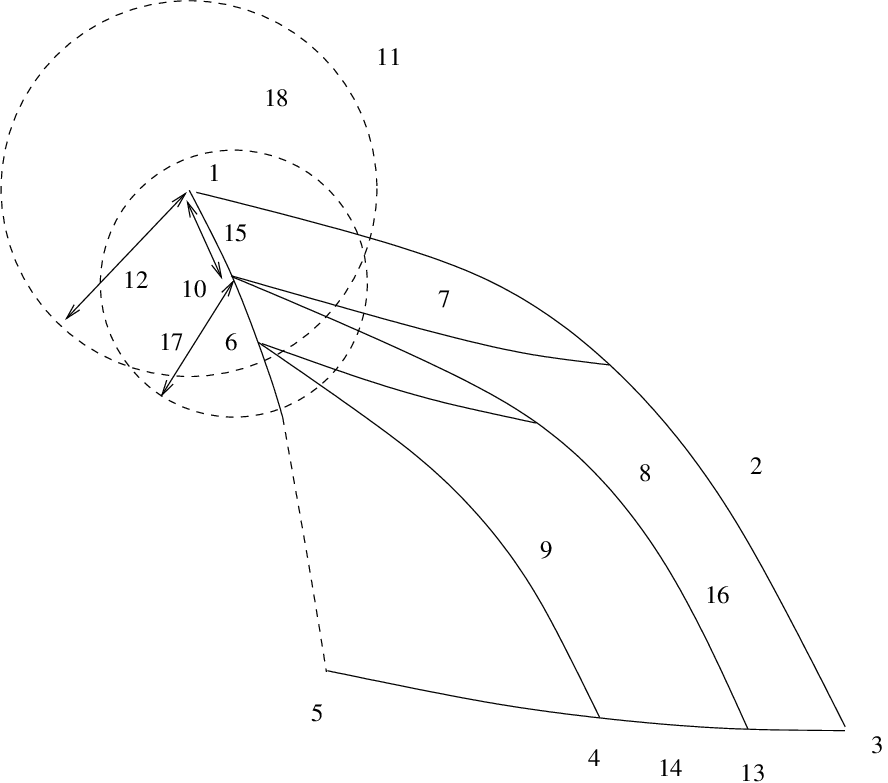}
\end{center}
\vspace*{-4mm}
\caption{Schematics of constructing the solution near $\Xi_3.$}
\label{construction}
\end{figure}

Therefore we finally establish the existence result in the entire transient wave region $\Rone$.
Lemma~\ref{excorners} provides the schematics, as illustrated in Figure~\ref{construction}, that is, how to construct the solution near $\Xi_1$ and $\Xi_3$ by selecting the correct data on $\Gamma_{23}$ to hold the compatibility conditions at $\Xi_1$ and $\Xi_3$. Furthermore  Lemmas~\ref{exlemma}, \ref{excorners} utilize the properties of the tricomi type degeneracy, that is, the sonic boundary cannot be a characteristic curve, and as a consequence $R=S=p_\theta$ remain strictly positive. 
Hence the natural compatibility condition where the different types of boundaries meet is $R=S=0$ which gives rise to an additional degeneracy. 
%Otherwise as in the proof of Lemma~\ref{excorners} we can extend the transient wave region $\Rone$.     

%%%%%%%%%%%%%%%%%%%%%%%%%%%%%%%%%
\section{Simple wave in region $\mathcal{R}_2$}
\label{sectionR_2}
We are now left with the simple wave region $\Rtwo$.
We proceed to derive certain identities that may prove useful. Note that from $$\frac{dr_+}{d\theta}=\lambda=r_+\sqrt{\frac{r^2_+ -c^2}{c^2}},$$ we can write
\begin{align*}
 c^2(p)=\gamma p^{\kappa}= \frac{r^4_+}{r '^2_++r^2_+},
\end{align*}
which implies
\begin{align*}
r^2_+-\gamma p^{\kappa}&=\frac{(r'_+r_+)^2}{r '^2_++r^2_+}.
\end{align*}
Differentiating along the positive characteristics $r=r_+$ we have
\begin{align*}
\partial_{+}(\gamma p^{\kappa})&=\frac{2r^3_+r'_+[r^2_++2r'^2_+-r_+r''_+]}{(r'^2_++r^2_+)^2}, 
\end{align*}
and consequently 
\begin{align}
R=\partial_{+}p&=\frac{2r^3_+r'_+p^{1/\gamma}}{\gamma \kappa}\left[\frac{r^2_++2r'^2_+-r_+r''_+}{(r'^2_++r^2_+)^2}\right]. \label{R1}
\end{align}
Similarly from $$\frac{dr_-}{d\theta}=- \lambda=-r_-\sqrt{\frac{r^2_- -c^2}{c^2}},$$ we have
\begin{align}
S=\partial_{-}p&=\frac{2r^3_-(-r'_-)p^{1/\gamma}}{\gamma \kappa}\left[\frac{r_-r''_--r^2_--2r'^2_-}{(r'^2_-+r^2_-)^2}\right]. \label{S1}
\end{align}
In the simple wave region $\mathcal{R}_2$, by following Lemma~\ref{Lemma1_increasing}, only $S=\partial_-p>0$ transfers the data across $\Gamma_{23}$, while the other family $R=\partial_+p$ becomes zero.

Hence by using $r''_+=r_++\dfrac{2r'^2_+}{r_+}$ from \eqref{R1} with initial values $r_+(\theta_0)=r_0, r'_+(\theta_0)=r_0\sqrt{\frac{r_0^2-c^2(p_0)}{c^2(p_0)}},$
prescribed on $\Gamma_{23}$,
we can find the equation of the positive characteristic, $\Gamma_+=\{(r_+(\theta),\theta)\}$,
\begin{align}
r_+(\theta)=\frac{r_0}{\left(\sin \theta_0-\cos \theta_0\sqrt{\frac{r_0^2-c^2(p_0)}{c^2(p_0)}}\right)\sin \theta+\left(\sin \theta_0\sqrt{\frac{r_0^2-c^2(p_0)}{c^2(p_0)}}+\cos\theta_0\right)\cos \theta}. \label{3}
\end{align}
Solving for $c^2(p_0)$ gives
\begin{eqnarray}\nonumber
c^2(p_0)&=&\frac{r_+^2r_0^2[\sin^2\theta+\sin^2 \theta_0-2\sin^2\theta_0\sin^2\theta-2\sin\theta\cos\theta\sin\theta_0\cos\theta_0]}{r^2+r_0^2-2rr_0\sin\theta\sin\theta_0-2rr_0\cos\theta\cos\theta_0}\\
& = &\frac{(\eta\xi_0-\xi \eta_0)^2}{(\eta-\eta_0)^2+(\xi-\xi_0)^2}, \label{4}
\end{eqnarray}
where $\xi_0=r_0\cos \theta_0, \ \ \eta_0=r_0 \sin \theta_0.$

We further find the positive characteristic emanating from $\Xi_2$, denoted by $\Gamma_{24}$, by integrating 
\begin{align*}
\frac{dr}{d\theta}=r\sqrt{\frac{r^2-c^2(p_4)}{c^2(p_4)}},
\end{align*}
and thus
$$\Gamma_{24}: \ \ r=c(p_4)\sec\bigg(\theta+\arcsec\sqrt{\frac{c(p_1)}{c(p_4)}}-\arcsin\sqrt{\frac{c(p_4)}{c(p_1)}}\bigg).$$
Figure~\ref{shock} depicts the envelop formation by the simple wave \eqref{3}.
\begin{figure}[ht]
\psfrag{2}[][][0.7][0]{$\Gamma_{12}$}
\psfrag{1}[][][0.7][0]{$\Xi_1$}
\psfrag{15}[][][0.7][0]{$\eta=c(p_1)$}
\psfrag{16}[][][0.7][0]{$\eta=c(p_4)$}
\psfrag{3}[][][0.7][0]{$\Xi_2$}
\psfrag{5}[][][0.7][0]{$\Xi_3$}
\psfrag{6}[][][0.7][0]{$\sigma$}
\psfrag{7}[][][0.7][0]{$\mathcal{R}_0$}
\psfrag{11}[][][0.7][0]{$\mathcal{R}_2$}
\psfrag{12}[][][0.7][0]{$\mathcal{R}_1$}
\psfrag{8}[][][0.7][0]{$\Xi_4$}
\psfrag{10}[][][0.7][0]{$\Gamma_{24}$}
\psfrag{9}[][][0.7][0]{$\Sigma'$}
\psfrag{4}[][][0.7][0]{$\Gamma_{23}$}
\begin{center}
\includegraphics[height = 2.9in,width = 2in]{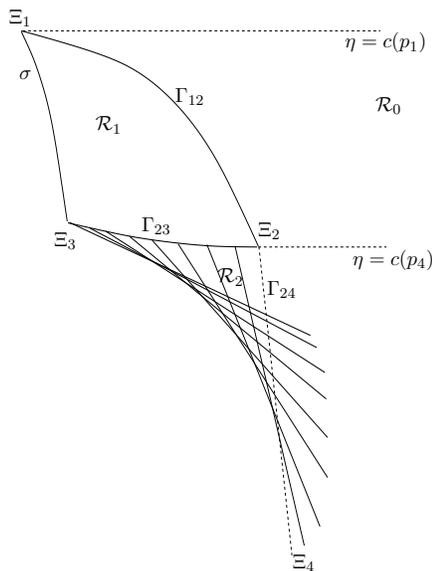}
\end{center}
\vspace*{-4mm}
\caption{Envelope formation by the simple wave.}
\label{shock}
\end{figure}

\section{Numerical Results}\label{numerics}
We conclude our paper by presenting the numerical results for the configuration. The results are produced by using the Riemann data $\rho_1=0.5$, $\rho_2=0.25$,
and $\gamma =3$. The computational domain that we have implemented is $10^{-2} \le r\le 1$ and $0\le \theta \le 3\pi/2$ where $(r,\theta)$ are polar coordinates, with mesh sizes $dr = 1/2400 \approx 4.1667\times10^{-4}$ and $d\theta = 2\pi/3600 \approx 1.7\times10^{-3}$, with the final time $T=1$.

These results are produced by using CLAWPACK \cite{Clawpack}.
We implement Roe average methods \cite{Roe} and finite volume methods on quadrilateral grids \cite{Clawpack}. 
More precisely, we implement Roe average methods in a uniform grid in polar coordinates as our computational domain, together with a coordinate mapping and 
appropriate scaling of the flux differences. 
The scaling is done by using the area ratio ``capacity'' of the computational cell which is determined by 
the size of the corresponding physical cell \cite{LeVeque1}.

\begin{figure}[ht]
\begin{center}
\includegraphics[height = 4in,width = 4in]{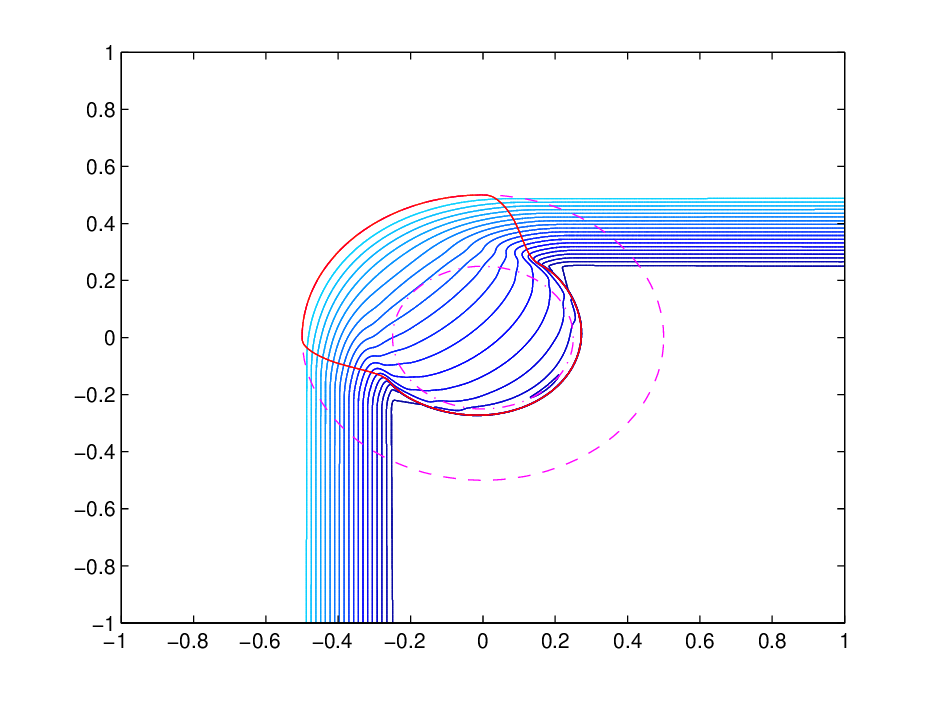}
\end{center}
\vspace*{-4mm}
\caption{Density plots: the contour plot of $\rho$.}
\label{fig_den}
\end{figure}

\begin{figure}[ht]
\begin{center}
\includegraphics[height = 2.5in,width = 2.5in]{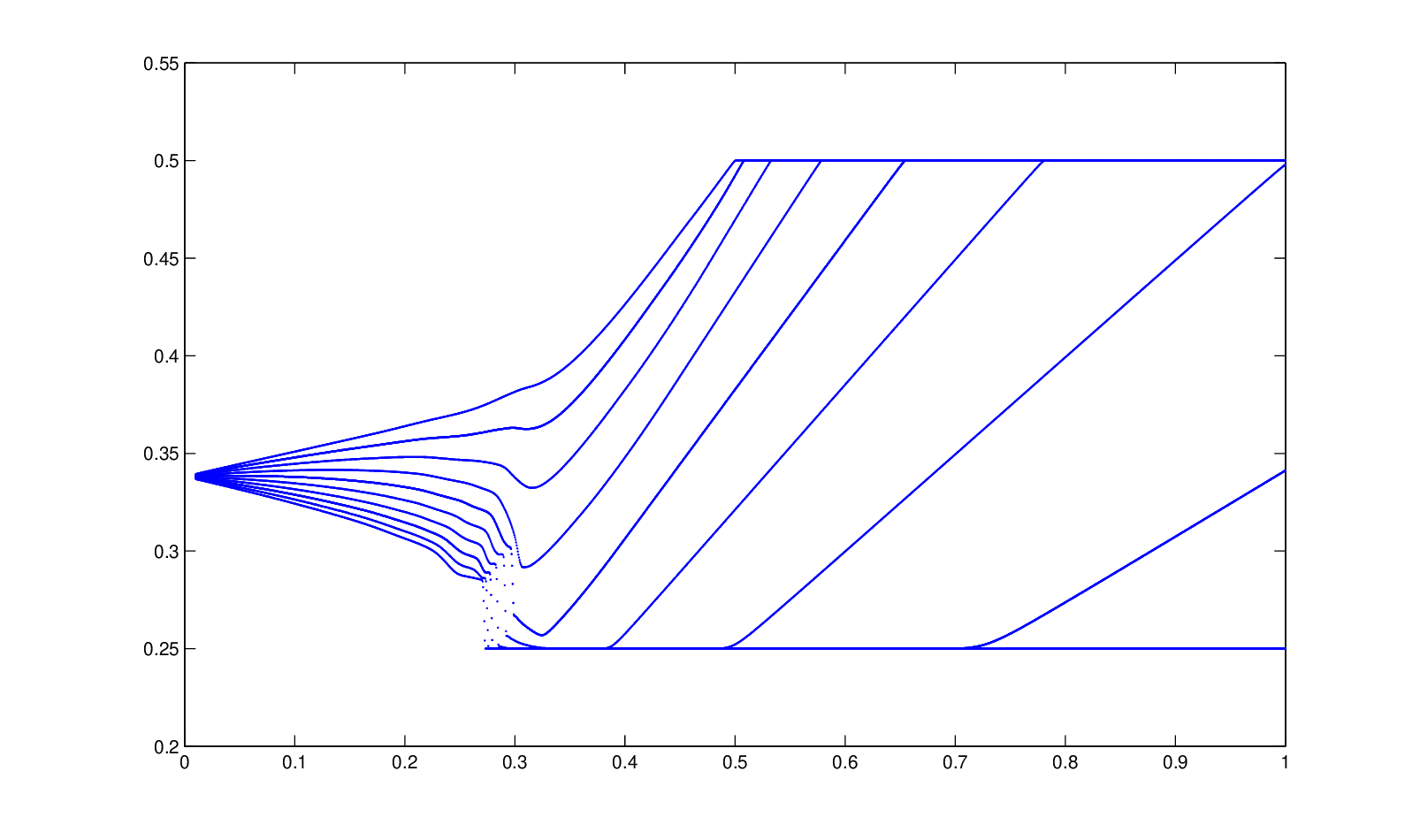}
\includegraphics[height = 2.5in,width = 2.5in]{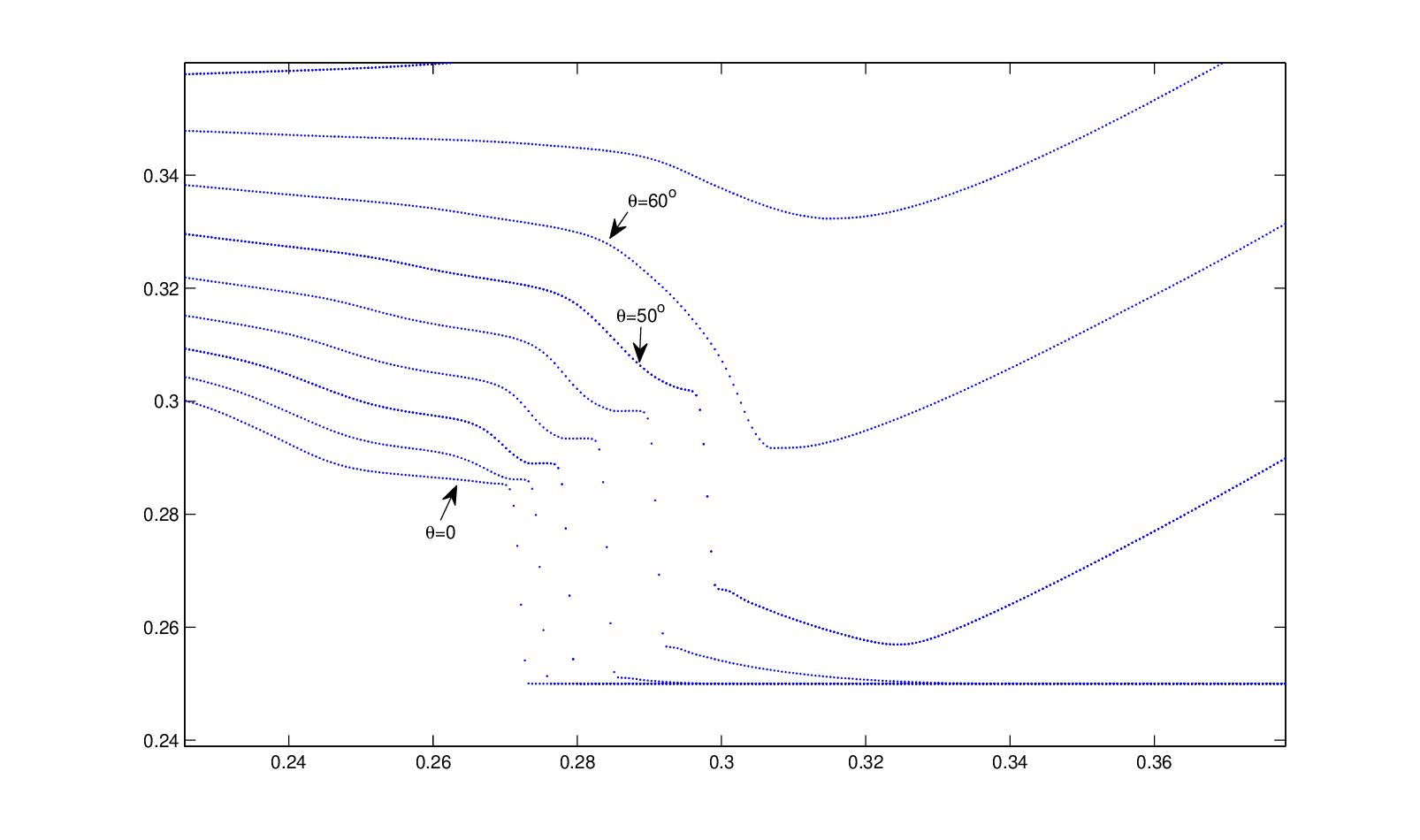}
\end{center}
\vspace*{-4mm}
\caption{Density plots: Left figure is the cross section in the radial direction for a fixed angle $\theta$ ranging from $0$ to $\pi/2$ and incrementing by $10$ degrees. Right figure is the enlargement of the left figure near which the shock changes to sonic, which appears to be in-between the angle $50$ and $60$ degrees in this configuration.  }
\label{fig_cross}
\end{figure}
In Figure~\ref{fig_cross}: the right figure is the enlargement from the left figure near which the shock appears. In the right figure, the density flattens out near the shock, while there exists a compression (a dip in the cross section) which merges to the shock.
The numeric suggests that the angle of the location of $\Xi_3$ where the sonic boundary and the shock boundary meet is between $50$ and $60$ degrees in this configuration. 

\section*{Acknowledgments} 
The first author is grateful to Yuxi Zheng and Kyungwoo Song for helpful discussions on the Pressure gradient system. The first author is also thankful to John Hunter for a discussion on Tricomi problems.

\end{document}